\documentclass[a4paper, 12pt, final]{amsart}

\makeatletter
\numberwithin{equation}{section}
\numberwithin{figure}{section}

\usepackage{amsmath,amsfonts,amssymb,amsthm,epsfig}

\usepackage{color}

\usepackage[final,colorlinks,linkcolor=blue,anchorcolor=red,citecolor=blue]{hyperref}
\usepackage[leqno]{amsmath}

\usepackage{amscd}
\usepackage{mathrsfs}
\usepackage{}
\usepackage[OT1]{fontenc}
\usepackage[latin1]{inputenc}
\usepackage{amssymb,latexsym}
\usepackage{type1cm,ae}
\usepackage[notref,notcite]{showkeys}
\usepackage{graphicx}
\usepackage{xspace}
\usepackage{setspace}
\usepackage[english]{babel}
\usepackage[all,cmtip]{xy}

\theoremstyle{definition}
\newtheorem{theorem}{Theorem}[section]
\newtheorem{proposition}[theorem]{Proposition}
\newtheorem{lemma}[theorem]{Lemma}

\newtheorem{corollary}[theorem]{Corollary}
\newtheorem{example}[theorem]{Example}

\newtheorem{definition}[theorem]{Definition}
\newtheorem{remark}[theorem]{Remark}
\numberwithin{equation}{section}

\newcommand*{\bR}{\ensuremath{\mathbb{R}}}

\newcommand*{\Wert}{\mathord{\mbox{|\kern-1.5pt|\kern-1.5pt|}}}
\newcommand*{\ie}{\mbox{i.e.}\xspace}

\newcommand{\R}{\mathbb{R}}	

\DeclareMathOperator{\diam}{diam}

\DeclareMathOperator{\adp}{adp}

\def\XXint#1#2#3{{\setbox0=\hbox{$#1{#2#3}{\int}$}
		\vcenter{\hbox{$#2#3$}}\kern-.5\wd0}}

\makeatother

\usepackage{babel}
\begin{document}
	
	\title[The Canonical Induced Metric for Topological Branched Coverings ]{The Canonical Induced Metric for Topological Branched Coverings and Geometric Applications}
	
	\author{Chang-Yu Guo}
	\date{\today}
	\address[Chang-Yu Guo]{Research Center for Mathematics and Interdisciplinary Sciences, Shandong University 266237,  Qingdao, P. R. China and Department of Physics and Mathematics, University of Eastern Finland, 80101, Joensuu, Finland} 
	\email{changyu.guo@sdu.edu.cn}
	
	\author{Marshall Williams}
	
	\thanks{C.-Y. Guo is supported by the Young Scientist Program of the Ministry of Science and Technology of China (No.~2021YFA1002200), the National Science Foundation of China (No.~12671095), the Taishan Scholar Project  and the Jiangsu Provincial Scientific Research Center of Applied Mathematics under Grant No. 7707014040A}

	\begin{abstract}
		Given a continuous map $f:X\to Y$ from a continuum-connected space into a metric space, we introduce an associated canonical pullback metric $d_f$ on $X$ by combining the target diameter structure with the Mazurkiewicz construction as 
		\[
		d_f(x,x')=\inf\{\operatorname{diam}_Y f(E): E\subset X\text{ is a continuum containing }x,x'\}.
		\]
		Our main result shows that if $f:(X,\rho)\to Y$ is $L$-BDD, then $d_f$ is quantitatively equivalent to the Mazurkiewicz metric $\rho_M$, $(d_f)_M=d_f$, every continuum satisfies $\operatorname{diam}_{d_f}E=\operatorname{diam}_Y f(E)$ and $f:(X,d_f)\to Y$ is $1$-BDD and hence $1$-BLD.  For bounded-turning spaces, this identifies a canonical bi-Lipschitz metric realization in the sense that every finite-to-one topological branched covering onto a locally length metric manifold admits such a compatible induced metric without any a priori BDD hypothesis.
		
		There are two main geometric applications of the induced pullback metric.  Firstly, we prove an index-free quantitative porosity theorem for branch sets and their images at points of controlled linear or inverse linear dilatation, extending the well-known Bonk-Heinonen theorem and Sarvas theorem for branch set of Euclidean quasiregular mappings to a large class of metric spaces.  Secondly, we prove that every finite-to-one BDD branched covering over a doubling bounded-turning target space admits a quantitative bi-Lipschitz embedding into the target times a finite-dimensional Euclidean space. As direct corollaries, we provide affirmative answers to an open question of Heinonen--Rickman [Duke Math. J. 2002] on Hausdorff dimension of the branch set of a BLD map between metric manifolds and to another open question on bi-Lipschitz embedding of a class of BLD Euclidean spaces, disprove a conjecture of Heinonen--Rickman [Duke Math. J. 2002] on null branch set of a BLD map into $\R^n$ and give a negative answer to the Heinonen--Semmes BLD-remetrization question [Conf. Geom. Dyn. 1997].
	\end{abstract}
	
	\maketitle
	\tableofcontents{}
	
	\section{Introduction}\label{sec:introduction}
	
	The starting point of this paper is the classical Mazurkiewicz metric introduced originally by S. Mazurkiewicz \cite{maz16}. Recall that if $(X,\rho)$ is a metric space, where points can be joined by continua, then its Mazurkiewicz metric $\rho_M$ is defined as
	\[
	\rho_M(x,x')=\inf\{\operatorname{diam}_{\rho}E:E\subset X\text{ is a continuum containing }x,x'\}.
	\]
	Thus $\rho_M$ measures the smallest continuum diameter needed to connect two points. As we will see later in this paper, for mappings with controlled distortion of continua, in particular, the new class of \emph{mappings of bounded diameter distortion} introduced below, this turns out to be the natural metric.  
	
	Given a continuous map $f:X\to(Y,d_Y)$, we pull back the \emph{diameter structure} of the target space $Y$ and define an induced pullback metric $d_f$ on $X$ as 
	\begin{equation}\label{eq:intro-canonical-df}
		d_f(x,x')=\inf\{\operatorname{diam}_{Y}f(E):E\subset X\text{ is a continuum containing }x,x'\}.
	\end{equation}
	Note that this metric is very different from the pointwise pullback $d_Y(f(x),f(x'))$, which is degenerate as soon as $f$ is noninjective.  
	
	\subsection{The canonical induced pullback metric}
	
	Before stating our main result on the induced pullback metric $d_f$, we first recall an important class of mappings, introduced by Martio and Vais\"al\"a \cite{mv88} in their study of elliptic equations. 
		A continuous map $f:(X,\rho)\to(Y,d_Y)$ between two metric spaces is said to be an \textit{$L$-BLD}, or a \textit{mapping of $L$-bounded length distortion}, $L\geq 1$, if 
		\begin{equation*}
			L^{-1}l(\alpha)\leq l(f\circ \alpha)\leq Ll(\alpha)
		\end{equation*}
		for all non-constant curves $\alpha$ in $X$, where $l(\gamma)$ denotes the length of a curve $\gamma$ in a metric space.
	The definition of BLD mappings is clearly only interesting if the metric spaces $X$ and $Y$ have a reasonable supply of rectifiable curves and so the most natural setting in which we study such mappings is that of quasiconvex metric spaces; recall that a metric space $X$ is \textit{$c$-quasiconvex} if every
	pair of points $x_1,x_2\in X$ may be joined with a curve $\gamma$ of length $l(\gamma)\leq cd(x_1,x_2)$. When the constant $c$ is unimportant, we omit it.

	Since our construction of the induced pullback metric $d_f$ generally only requires the control of diameters of sets, a more appropriate setting for our results will be the more general class of bounded turning metric spaces. Recall that a metric space $X$ has \textit{$c$-bounded turning} if every pair of points $x_1,x_2\in X$ can be joined by a continuum $E\subset X$ such that $\diam E\leq cd(x_1,x_2)$. For locally path connected metric spaces, we may use curves instead of general continuua in the definition of bounded turning. It is elementary to verify that the infimum of the diameters of continuua or curves joining two points is realized, provided $X$ is assumed to be complete as well as locally compact.
	
	When $X$ and $Y$ have bounded turning, the natural branched analog of a bi-Lipschitz homeomorphism is the new class of \textit{mappings of bounded diameter distortion}, which is defined in analogy with BLD mappings:
	\begin{definition}\label{def:BDD mapping}
		A continuous map $f:(X,\rho)\to(Y,d_Y)$ between two metric spaces is said to be an \textit{$L$-BDD}, or a \textit{mapping of $L$-bounded diameter distortion}, $L\geq 1$, if 
		\begin{equation*}
			L^{-1}\operatorname{diam}_{\rho}E\le \operatorname{diam}_{Y}f(E)\le L\operatorname{diam}_{\rho}E
		\end{equation*}
		for every nondegenerate continuum $E\subset X$. 
	\end{definition}
	
	
	The following theorem shows that the induced pullback metric $d_f$ is canonical. 	Recall that a metric $\rho$ on a topological space $X$ is said to be \emph{compatible} if the topology induced by $\rho$ coincides with the given topology of $X$.
	
	\begin{theorem}[Canonical induced metric]\label{thm:intro-canonical-induced-metric}
		Let $f:X\to(Y,d_Y)$ be continuous, where $X$ is continuum connected, and let $d_f$ be defined by \eqref{eq:intro-canonical-df}.  Suppose  $\rho$ is a compatible metric on $X$ and $f:(X,\rho)\to(Y,d_Y)$ is $L$-BDD.  Then
		\begin{equation}\label{eq:intro-Maz-comparison}
			L^{-1}\rho_M\le d_f\le L\rho_M.
		\end{equation}
		In particular, $d_f$ is a metric.  Moreover, $(d_f)_M=d_f$ and for every continuum $E\subset X$, it holds 
		\begin{equation}\label{eq:intro-exact-diameter}
			\operatorname{diam}_{d_f}E=\operatorname{diam}_{Y}f(E).
		\end{equation}
		Consequently, $f:(X,d_f)\to(Y,d_Y)$ is $1$-BDD and hence $1$-BLD. 
		
		 If $(X,\rho)$ has $C$-bounded turning, then
		\begin{equation}\label{eq:metric comparison bounded turning}
		L^{-1}\rho\le d_f\le LC\rho,
		\end{equation}
		and so $d_f$ is quantitatively bi-Lipschitz equivalent to the original metric.  Moreover, $d_f$ is the unique compatible Mazurkiewicz-fixed metric that makes $f$ $1$-BDD.
		
		If $f:M\to N$ is  a finite-to-one topological branched covering between connected metric manifolds and $N$ is locally a length space, then formula \eqref{eq:intro-canonical-df} defines a compatible metric on $M$, which makes $f:(M,d_f)\to N$ $1$-BDD and $1$-BLD.  If $N$ and $f$ are proper, then $(M,d_f)$ is proper.
	\end{theorem}
	
	Recall that a continuous mapping $f\colon X\to Y$ between topological spaces is said to be a topological \textit{branched covering} if $f$ is \textit{discrete} and \textit{open}, \ie $f$ is an open mapping and for each $y\in Y$ the preimage $f^{-1}(y)$ is a discrete subset of $X$.
	Thus Theorem \ref{thm:intro-canonical-induced-metric} shows that the induced pullback metric is not merely a normalization available after BDD geometry has already been given.  It is intrinsic to the pair $(f,d_Y)$ and provides a canonical metric realization of every finite topological branched covering in the manifold setting.  
	
	The comparison \eqref{eq:intro-Maz-comparison} also explains exactly where bounded turning enters: BDD mpa controls diameters of continuum and so we have the intrinsic bi-Lipschitz comparison with $\rho_M$. Equation \eqref{eq:metric comparison bounded turning} shows that bounded turning of $X$ is what turns that ``diameter comparison" into a ``metric comparison".  
	
	Using the induced pullback metric, we obtain a very useful adapted pullback factorization 
	\[
	f:(X,d_X)\xrightarrow{\mathrm{id}}(X,d_f)\xrightarrow{f_{\rm adp}}(Y,d_Y),
	\]
	where $f_{\rm adp}:(X,d_f)\to (Y,d_Y)$ is the same set map as $f$. This adapted map is always 1-Lipschitz, 1-BDD and 1-BLD. In Section \ref{sec:the pullback factorization}, we will develop a rather rich theory of induced metrics, the adapted pullback factorization and demonstrate this is the canonical metric attached to branched coverings. These will be important in our geometric applications.
	
	
	\subsection{First geometric application: quantitative porosity of branch sets}
	The first application of the canonical induced pullback metric is quantitative control on the branch set of BLD mappings. Recall that the \textit{branch set} $\mathcal{B}_f$ of a mapping $f\colon X\to Y$ between topological spaces is the closed set of points in $X$ where $f$ does not define a local homeomorphism. In the case that $X$ and $Y$ are \textit{generalized $n$-manifolds}, $\mathcal{B}_f$ can be interpreted alternatively as the set of points at which the \textit{local index} $i(x,f)=1$.
	
	By a result of due to Chernavski\u{\i} and V\"ais\"al\"a~\cite{v66}, \textit{the branch set $\mathcal{B}_f$ of a branched covering $f\colon X\to Y$ between $n$-manifolds has topological dimension at most $n-2$.} Moreover, for a branched covering $f\colon \bR^n\to \bR^n$, both $\dim_{\mathcal{H}}(\mathcal{B}_f)$ and $\dim_{\mathcal{H}}f\big(\mathcal{B}_f\big)$ can be equal to $n$; see also \cite{ap17,KLT21,KLT22} for some recent results. Thus towards Hausdorff dimensional estimates of the branch set $\mathcal{B}_f$, further analytic assumptions have to be imposed on the branched coverings $f\colon X\to Y$. The common classes of mappings that arose great interests in the past two decades are the so-called \textit{quasiregular mappings}, or \textit{mappings of bounded distortion}; see~\cite{mrv69,mrv70,mrv71,re89,r93} for the general theory of quasiregular mappings.
	
	For a branched covering $f\colon X\to Y$ between two metric spaces, $x\in X$ and $r>0$, set
	\begin{equation*}
		H_f(x,r)=\frac{L_f(x,r)}{l_f(x,r)},
	\end{equation*}
	where
	\begin{equation*}
		L_f(x,r):=\sup{\{d(f(x),f(y)):d(x,y)= r\}},
	\end{equation*}
	and
	\begin{equation*}
		l_f(x,r):=\inf{\{d(f(x),f(y)):d(x,y)= r\}}.
	\end{equation*}
	Then the \textit{linear dilatation function} of $f$ at $x$ is defined pointwise by 
	\begin{equation*}
		H_f(x)=\limsup_{r\to0}H_f(x,r).
	\end{equation*}
	Here, we make the convention  $H_f(x,r)=\infty$ when $l_f(x,r)=0<L_f(x,r)$.
	
	A mapping $f\colon X\to Y$ between two metric measure spaces is termed \textit{(metrically) $H$-quasiregular} if the linear dilatation function $H_f$ is finite everywhere and essentially bounded from above by $H$. We call $f$ a quasiregular mapping if it is $H$-quasiregular for some $H\in [1,\infty)$. The branch set of a quasiregular mapping can be very wild, for instance, it might contain many wild Cantor sets, such as the Antoine's necklace~\cite{hr98}, of classical geometric topology. In his 2002 ICM address~\cite[Section 3]{h02}, Heinonen asked the following question: \textit{Can we describe the geometry and topology of allowable branch sets of quasiregular mappings between metric $n$-manifolds?} 
	
	Addressing the above question, we shall explore the (geometric) porosity of $\mathcal{B}_f\cap A$ and $f\big(\mathcal{B}_f\cap A\big)$ when the linear dilatation of $f$ is finite on $A$. Our main result states that if $X$ satisfies a quantitative local connectivity assumption, then the aforementioned sets are quantitatively porous. This leads to quantitative Hausdorff dimensional estimates of these sets.
	
	Regarding the Hausdorff dimension of $\mathcal{B}_f$ and its image $f\big(\mathcal{B}_f\big)$ in the Euclidean setting, a well-known result of Gehring and V\"ais\"al\"a~\cite{gv72} says that for each $n\geq 3$ and each pair of numbers $\alpha,\beta\in [n-2,n)$, there exists a quasiregular mapping $f:\bR^n\to \bR^n$ such that 
	\begin{equation*}
		\dim_{\mathcal{H}}\mathcal{B}_f=\alpha\quad \text{and}\quad \dim_{\mathcal{H}}f(\mathcal{B}_f)=\beta.
	\end{equation*}
	On the other hand, by the result of Sarvas~\cite{s75}, for a non-constant $H$-quasiregular mapping $f\colon \Omega\to \bR^n$, $n\geq 2$, between Euclidean domains, 
	\begin{equation}\label{eq:Sarvas}
		\dim_{\mathcal{H}}f(\mathcal{B}_f)\leq n-\eta
	\end{equation}
	for some constant $\eta=\eta(n,H)>0$. 	
	A well-known result of Bonk and Heinonen~\cite{bh04} says that for a non-constant $H$-quasiregular mapping $f\colon \Omega\to \bR^n$, $n\geq 2$, between Euclidean domains, 
	\begin{equation}\label{eq:Bonk-Heinonen}
		\dim_{\mathcal{H}}\mathcal{B}_f\leq n-\eta
	\end{equation}
	for some constant $\eta=\eta(n,H)>0$. Let us point out that the proof of Bonk and Heinonen~\cite{bh04} relies on an earlier theorem of Sarvas~\cite{s75}, which implies the existence of a quantitative upper bound on $\dim_{\mathcal{H}}(\{x\in \bR^n: 2\leq i(x, f )\leq m\})$ below $n$, depending only on $H$, $n$, and $m$. The Bonk--Heinonen theorem then follows upon proving that $\{x\in \bR^n:i(x,f)>m\}$ is porous, for some $m$ depending only on $n$ and $H$.
	
	Since the Sarvas theorem used a normal family argument, the dimension bound obtained
	in~\eqref{eq:Bonk-Heinonen} was not directly computable; Onninen and Rajala~\cite{or09} later proved that the sets $\{x\in \bR^n: i(x,f)\leq m\}$ are $\delta_m$-porous, as are their images under $f$, with a directly computable porosity constant $\delta_m$, which combined with the Bonk--Heinonen porosity result for points of large index, gives a computable but index-dependent dimension bound on $\dim_{\mathcal{H}}\mathcal{B}_f$.
	
	On the other hand, very little is known about the Hausdorff dimension of $\mathcal{B}_f$ and its image $f\big(\mathcal{B}_f\big)$ for quasiregular mapping $f\colon X\to Y$ beyond the Euclidean spaces. Indeed, in the non-smooth setting, all the known proofs of the fact that both the branch set and its image are null sets with respect to an Ahlfors regular measure relies on certain (Lipschitz) differentiable structure akin to the Euclidean spaces. For instance, Heinonen and Rickman~\cite{hr02} have established a general theory of BLD mappings between the so-called \textit{generalized manifolds of type A}, on which both the branch set and its image are null sets with respect to the Ahlfors reulgar measures, and put it as an open problem~\cite[Remark 6.7 (b)]{hr02} whether it is possible to obtain a quantitative estimate as~\eqref{eq:Bonk-Heinonen}. Note that a quantiative estimate as~\eqref{eq:Bonk-Heinonen} can be applied to improve on the main result of~\cite{hs02}, regarding the size of the exceptional set of the bi-Lipschitz parametrization.
	For quasiregular mappings from the Euclidean domain to generalized manifolds of type $A$, Onninen and Rajala~\cite{or09} were able to obtain a slightly weaker estimate of the form~\eqref{eq:Sarvas}, while an estimate of the form~\eqref{eq:Bonk-Heinonen} seems to be un-reachable.
	
	For each $R>0$ and $H\geq 1$, set 
	\begin{equation*}
		S_{H,R}=\{x\in X:H_f(x,r)\leq H\quad \text{for all } r<R \}
	\end{equation*}
	and
	\begin{equation*}
		S_{H,R}^*=\{x\in X:H_f^*(x,r)\leq H\quad \text{for all } r<R \},
	\end{equation*}
	where $H_f^*$ is the inverse dilatation function as defined in Section~\ref{subsec:normal-neighborhoods}.
	
	Our main result below shows that most points where the dilatation or inverse dilatation is finite are not branch points. 
	\begin{theorem}\label{thm:thm branch set}
		Let $X$ and $Y$ be locally compact, locally doubling metric generalized $n$-manifolds. Assume that $X$ is locally linearly locally $n$-connected and that $Y$ has bounded turning. Let $f:X\to Y$ be a branched covering. Then for each $R>0$ and $H\geq 1$, the sets $S_{H,R}\cap\mathcal B_f$, $f(S_{H,R}\cap\mathcal B_f)$, $S^*_{H,R}\cap\mathcal B_f$, and $f(S^*_{H,R}\cap\mathcal B_f)$ are quantitatively porous.  The porosity constants depend only on $H$ and the quantitative local data of $X$ and $Y$.
	\end{theorem}
	
	In the special case that $X$ and $Y$ are Euclidean spaces, Theorem~\ref{thm:thm branch set} gives a nice decomposition of the the branch set of a branched covering into countable union of sets restricted to which the branch set is quantitatively porous. Thus, it can be regarded as a strengthened version of the earlier quantitative porosity results of Bonk--Heinonen~\cite{bh04} and Onninen--Rajala~\cite{or09} for the branch set of a quasiregular mapping. Moreover, the quantitative porosity bounds on $f(\mathcal{B}_f)$ seems to be new even for a quasiregular mapping $f\colon \bR^n\to \bR^n$ and it can be regarded as a strengthened version of the dimensional estimate of Sarvas~\cite{s75}.
	
	Write
	\begin{equation*}
		S_H=\cup_{R>0}S_{H,R},\quad S_H^*=\cup_{R>0}S_{H,R}^*,
	\end{equation*}
	and
	\begin{equation*}
		S_f=\cup_{H<\infty}S_H,\quad S_f^*=\cup_{H<\infty}S_H^*.
	\end{equation*}
	As an immediate application of Theorem~\ref{thm:thm branch set}, we obtain the following corollary. 
	
	\begin{corollary}\label{coro:Corollary 1.2}
		For all locally doubling measures $\mu$ on $X$ and $\nu$ on $Y$, 
		\begin{equation*}
			\mu\big(S_f\cap\mathcal{B}_f\big)=\nu\big(f(S_f\cap\mathcal{B}_f)\big)=
			\mu\big(S_f^*\cap\mathcal{B}_f\big)=\nu\big(f(S_f^*\cap\mathcal{B}_f)\big)=0.
		\end{equation*}
		In particular, if either $H_f(x)<\infty$ or $H_f^*(x)<\infty$ for all $x\in X$, then $$\mu(\mathcal{B}_f)=\nu(f(\mathcal{B}_f))=0.$$
	\end{corollary}
	
	
	Particularly important to the general theory of quasiconformal and quasisymmetric mappings are Ahlfors $Q$-regular spaces. It is well know that porous subsets of such spaces have Hausdorff dimension strictly smaller than $Q$, quantitatively; see e.g.~\cite[Lemma 9.2]{or09}. Thus we have the following immediate consequence.
	\begin{corollary}\label{coro:Corollary 1.3}
		If, in addition, $X$ and $Y$ are Ahlfors $Q$-regular and either $H_f\le H$ or $H_f^*\le H$ everywhere, then
		\[
		\max\{\dim_{\mathcal H}\mathcal B_f,\dim_{\mathcal H}f(\mathcal B_f)\}\le Q-\eta<Q,
		\]
		where $\eta>0$ depends only on $H$ and the data.
	\end{corollary}
	
	Since BDD mappings have bounded inverse linear dilatation (see Lemma~\ref{lemma:BDD-inverse-dilatation}), we obtain the following corollary for BDD mappings. 
	\begin{corollary}\label{coro:BDD-porosity}
		In the setting of Theorem~\ref{thm:thm branch set}, assume in addition that $X$ has locally uniformly bounded turning. If $f$ is locally $L$-BDD, then $H_f^*$ is locally bounded quantitatively in terms of $L$ and the bounded-turning constant. Hence $\mathcal B_f$ and $f(\mathcal B_f)$ are countably quantitatively porous. In particular, in the Ahlfors $Q$-regular case, they have Hausdorff dimension at most $Q-\eta$.
	\end{corollary}
	
	As another useful application of Theorem~\ref{thm:thm branch set}, we obtain the quantitative Hausdorff dimension bound for branch set and its image of a BLD map between generalized manifolds of type A, which answers affirmatively the open problem of Heinonen and Rickman~\cite[Remark 6.7 (b)]{hr02}.
	\begin{corollary}\label{coro:Heinonen-Rickman}
		Let $f:X\to Y$ be an $L$-BLD branched covering between generalized $n$-manifolds equipped locally with complete length metrics, with $X$ locally linearly locally $n$-connected and $X,Y$ locally doubling. Then $H_f$ is locally bounded by $L^2$ and so the branch set and its image are quantitative porous and have Hausdorff dimension quantitatively less than $n$. 
	\end{corollary}
	
	The local linear $n$-connectivity is necessary for the validity of  the previous result, as the following example  indicates.
	
	\begin{theorem}\label{thm:counter-example}
		For every integer $n\ge3$, there exist a proper metric space $X$, which is homeomorphic to $\mathbb R^n$, and a degree-two branched covering
		\[
		f:X\to\mathbb R^n
		\]
		with the following properties: $f$ is $1$-BDD and $1$-BLD; $X$ is locally Ahlfors $n$-regular, supports a $(1,1)$-Poincar\'e inequality and admits a bi-Lipschitz embedding into some Euclidean space, while
		\[
		\min\bigl\{\mathcal H^n_X(\mathcal B_f),\mathcal H^n_{\mathbb R^n}(f(\mathcal B_f))\bigr\}>0.
		\]
	\end{theorem}
	
	Theorem~\ref{thm:counter-example} provides a counter-example to the following  conjecture of Heinonen and Rickman~\cite[Remark 6.32 (b)]{hr02}: \textit{if $X$ be a locally BLD $n$-Euclidean space that is locally bi-Lipschitz embeddable to some Euclidean space, then for all (Lipschitz) BLD-maps $f\colon X\to \R^n$, the branch set $\mathcal{B}_f$ has zero Hausdorff $n$-measure.} 
	
	As we will see shortly, the proof of Theorem~\ref{thm:counter-example} relies crucially on the induced pullback metric construction: starting from a degree-two topological branched covering $f$ with $\mathcal H^n_{\mathbb R^n}(f(\mathcal B_f))>0$, the induced pullback metric $d_f$ supplied by Theorem~\ref{thm:intro-canonical-induced-metric} turns $f$ into a $1$-BDD/$1$-BLD map without changing its topological branching set.
	
	\subsection{Second geometric application: bi-Lipschitz embedding}
	A longstanding question in geometric analysis asks when a given metric space admits a bi-Lipschitz embedding into some finite dimensional Euclidean spaces (see e.g.~\cite[Section 16.4]{h07}). However, in general, it turns out to be a  very difficult problem to find nontrivial naturally bi-Lipschitz invariant sufficient criteria for the embeddability; see~\cite{ds97,s93,s96b,s96c,s99,h03} and the references therein for both positive and negative results along this direction.
	
	As second application of the induced pullback metric, we consider the bi-Lipschitz embedding problem by assuming a priori the existence of a BDD covering with bounded multiplicity. This kind of assumption was imposed earlier by Heinonen and Rickman~\cite{hr02} in their study of geometric branched coverings and later it played an important role in the seminal work of Heinonen--Sullivan~\cite{hs02} and Heinonen--Keith~\cite{hk11} to obtain local bi-Lipschitz parametrization of certain metric spaces.

	
	Our main result is the following bi-Lipschitz embedding theorem for BDD coverings with finite multiplicity.
	\begin{theorem}\label{thm:thmB}
		Let $f:X\to Y$ be an $L$-BDD branched covering with finite multiplicity $N=N(f,X)$.  Suppose that $X$ and $Y$ have bounded turning and that $Y$ is doubling.  Then there are an integer $m\le c_d(N-1)$ and a quantitative bi-Lipschitz embedding
		\[
		\Psi:X\to Y\times\mathbb R^m,
		\]
		where $c_d$ depends only on the doubling and bounded-turning data of $Y$.  In particular, finite-multiplicity BLD coverings over Euclidean space have locally Euclidean bi-Lipschitz embeddings.
	\end{theorem}
	
	An immediate corollary of Theorem~\ref{thm:thmB} (see e.g.~Corollary~\ref{coro:biLip embedding II}) is that \emph{every BLD covering of a locally compact quasiconvex metric space over $\R^n$ with finite maximal multiplicity embeds bi-Lipschitzly into some Euclidean space}. This answers in the affirmative an open question of Heinonen and Rickman~\cite[Remark 6.32 (a)]{hr02}.

	In~\cite[Question 27]{hs97}, Heinonen and Semmes asked \textit{if for a given branched covering $f\colon S^n\to S^n$, $n\geq 3$, there is a metric $d$ on $S^n$ so that $(S^n ,d)$ is an Ahlfors $n$-regular and locally linearly contractible metric space, and $f\colon (S^n , d)\to  S^n$ is a BLD mapping.} 
	
	The canonical induced pullback metric shows that the BLD part of this problem has no obstruction in our setting: every finite topological branched covering admits a canonical compatible metric $d_f$ for which the same map is $1$-BDD and hence $1$-BLD.  Thus the content of \cite[Question 27]{hs97} is concentrated in the additional regularity requirements imposed on the domain metric.
	
	By Corollary~\ref{coro:Heinonen-Rickman}, the existence of such a metric $d$ necessarily implies that $f(\mathcal{B}_f)$ must be null with respect to the $n$-dimensional Hausdorff measure $\mathcal{H}^n$. On the other hand, we give an negative answer to this question by constructing an counter-example below.
	
	\begin{corollary}\label{coro:Heinonen-Semmes}
		For every $n\ge3$ there is a branched covering $F:S^n\to S^n$ for which no compatible metric $d$ on the domain can simultaneously make $(S^n,d)$ Ahlfors $n$-regular and linearly locally contractible and make
		\[
		F:(S^n,d)\to S^n
		\]
		a BLD mapping. 
	\end{corollary}
	
	There is also an interesting connection with \cite[Question 28]{hs97}, where the authors asks \emph{whether every branched covering of $S^n$ is topologically conjugate to a quasiregular mapping}.  Our results  provide a canonical intermediate model: every finite branched covering can first be represented as a $1$-BLD map
	\[
	f:(S^n,d_f)\to S^n.
	\]
	Consequently, one may view \cite[Question 28]{hs97} as asking whether the canonical metric sphere $(S^n,d_f)$ can be quantitatively conjugate back to the standard sphere in a way compatible with the branched covering.  This reduces the problem of constructing a controlled branched covering metric to a uniformization problem for the induced pullback metric.
	
	\subsection{Novelty of our approach}
	In this section, we briefly indicate the novelty of this work. The paper has three main  parts, linked together by the canonical metric $d_f$, but each addresses a different aspect of the geometry of branched coverings. In summary, 
	\[
	\text{Canonical induced metric}
	\quad\Longrightarrow\quad
	\begin{cases}
		\text{quantitative control of branching},\\
		\text{quantitative bi-Lipschitz embedding},\\
		\text{quasiregular mappings in metric spaces}.
	\end{cases}
	\]
	The first part builds the general theory of induced pullback metrics. Then the second and third show that the theory we established has concrete geometric consequences in two different directions. The last application concerns the theory of quasiregular mappings in metric spaces, in parallel to the seminal work of Heinonen-Koskela \cite{hk98} for quasiconformal maps. However, due to very different analytic flavors, it will be studied in a separate companion paper \cite{gw-2026}. 
	
	We emphasize below what is new in each part and how it differs from the classical results on which it builds.
	\smallskip 

	(1) The classical Mazurkiewicz construction replaces point distance by the least diameter of a continuum joining two points. Our first new observation is that, for a possibly noninjective map $f:X\to Y$, the target diameter structure can be pulled back before performing the Mazurkiewicz construction. This produces
		\[
		d_f(x,x')=\inf_{E\ni x,x'}\operatorname{diam}_Y f(E),
		\]
		which remains extremely meaningful in the branched situation, where the pointwise pullback $d_Y(f(x),f(x'))$ might be degenerate. The novelty is not merely the definition of an auxiliary metric, but the rigidity properties it enjoys: for every continuum $E$,
		\[
		\operatorname{diam}_{d_f}E=\operatorname{diam}_Y f(E),
		\qquad (d_f)_M=d_f,
		\]
		and every $L$-BDD realization satisfies $L^{-1}\rho_M\le d_f\le L\rho_M$. Thus the construction singles out a canonical Mazurkiewicz-fixed representative of the diameter geometry of $f$. As Theorem \ref{thm:intro-canonical-induced-metric} (or equivalently Theorem \ref{thm:canonical-diameter-metric}) and Corollary \ref{coro:canonical-two-metrics} demonstrate, in the bounded-turning case, this representative is quantitatively bi-Lipschitz equivalent to every BDD realization, while for a finite topological branched covering between metric manifolds the same construction produces a compatible $1$-BDD and thus $1$-BLD, without assuming any a priori BDD structure. In this sense, we could say that the Mazurkiewicz construction here should be regarded as a remetrization principle for branched coverings, rather than merely as a metric associated to a pre-existing domain. 
	\smallskip 
	
	(2) The second contribution is a quantitative index-free porosity theory. In the Euclidean quasiregular theory, Bonk--Heinonen \cite{bh04} and Onninen--Rajala \cite{or09} obtained quantitative porosity information for the branch set, while Sarvas \cite{s75} obtained dimensional bound for the image. Our Theorem~\ref{thm:thm branch set} treats the branch set and the branch image simultaneously, applies in a substantially broader metric-manifold setting and is formulated only in terms of quantitative local connectivity together with bounds for the linear or inverse linear dilatation. In particular, the porosity constants are independent of the local multiplicity. The proof is purely topological and metric: it combines the quantitative topology technique developed by Newman--McAuley--Robinson \cite{mr83} and the nerve-and-extension machinery developed by Grove--Petersen--Wu/Petersen \cite{gp88,gpw90,p90} (see also \cite{s96}), rather than the classical differentiability + local-degree arguments. This yields quantitative Hausdorff-dimension upper bound in Ahlfors regular spaces and answers the Heinonen--Rickman dimension question for BLD mappings. The counter-example in Theorem~\ref{thm:counter-example} also shows that the connectivity hypotheses are not artificial: even a degree-two $1$-BLD map from a locally Ahlfors regular PI space admitting a Euclidean bi-Lipschitz embedding may have branch set and branch image of positive $n$-measure.
	
	\smallskip
	(3)	The third contribution is a quantitative bi-Lipschitz embedding theorem for finite-to-one branched coverings. After replacing the domain metric by the induced pullback metric $d_f$, the map becomes $1$-BDD and continuum diameters in the induced metric agree exactly with image diameters in the target. The remaining obstruction to injectivity is therefore purely the separation of finitely many sheets. We resolve this by introducing quantitative component-separation scales and a finite coloring construction, which produces Euclidean coordinate functions separating points in common fibers while retaining the target $Y$ itself as the principal coordinate. The result is the embedding
		\[
		X\hookrightarrow Y\times\mathbb R^m,
		\qquad m\le c_d(N-1),
		\]
		with quantitative control in terms of the doubling, bounded-turning, BDD, and multiplicity data. This is different in nature from a general embedding theorem for doubling spaces: it exploits the branched covering structure and gives an embedding compatible with the original map. As consequences, it answers the Heinonen--Rickman bi-Lipschitz embedding question for finite-multiplicity BLD Euclidean spaces. Moreover, when combined with the positive measure branch set construction and our earlier porosity theorem, yields a negative answer to the Heinonen--Semmes BLD-remetrization question.

	\smallskip
	\textbf{Structure.} This paper is organized as follows. Section \ref{sec:preliminaries} contains the necessary preliminary results needed for later sections. In Section~\ref{sec:the pullback factorization}, we develop the theory of the canonical induced pullback metric $d_f$ via the Mazurkiewicz metric and diameter structures, and then study its universal realization property and compare it with the BLD geometry. In Section~\ref{sec:branch set of a qr map}, we apply the theory of induced pullback metric to study geometric porosity of BLD mappings and prove Theorem~\ref{thm:thm branch set}. In Section~\ref{sec:bi-Lipschitz embeddability}, we proves the quantitative bi-Lipschitz embedding theorem by exploiting the normalized $1$-BDD geometry and a finite coloring/separation construction. 
	
	\smallskip 
	
	\textbf{Acknowledgements.} We would like to express our gratitude to Prof.~Pekka Koskela
	and Prof.~Kai Rajala for their constant encouragement throughout this work and for
	their useful comments. We are also grateful to Prof.~Pekka Pankka for his interest in
	our work and for his useful feedback. Last but not least, we would like to thank many
	of our mentors and colleagues for their valuable discussions.
	
	The main body of this work and  the companion work \cite{gw-2026} were essentially completed ten years ago. However, due to various reasons (particularly the terrible exposition and extremely loose proof presentation), it failed a few times in journal submission. We feel deeply sorry about our irresponsible way of writing. The current paper provides detailed proofs of the claimed result in \cite{gw16paris}.  
	
	The authors are deeply indebted to an anonymous referee, who wrote an extremely kind and helpful long report in 2023, which contains numerous constructive suggestions, insightful comments and important corrections. It is that report, and constant encouragement from colleagues of the first author (thus thanks spread to Prof.~Xiao Zhong, Nageswari Shanmugalingam, Jeremy Tyson, Chang-Lin Xiang, Yuan Zhou, Renjin Jiang, Manzi Huang and many others),  motivate the (first) author to finally complete the revision in nearly three years. 
	
	The induced pullback metric was genuinely introduced by the second author, and during the revision, the theory of induced pullback metric turned ou to be unexpectedly rich, which probably surpassed the expectation of previous kind referee. The first author of this paper, sincerely hope that one day, the second author would come back to the quasiconformal world, and see how his genuinous idea was widely applied.

	\section{Preliminaries}\label{sec:preliminaries}
	\subsection{Generalized manifolds, local degree and local index}
	Let $H_c^*(X)$ denote Alexander--Spanier cohomology with compact supports and integer coefficients. A finite-dimensional space $X$ is a generalized $n$-manifold if its local cohomology is that of $\mathbb R^n$. We use the standard local-degree theory for oriented generalized manifolds; see~\cite{hr02}.
	
For a branched covering $f:X\to Y$ and $x\in X$, define the local multiplicity index by
		\[
		i(x,f)=\inf\{N(f,U):U\text{ is an open neighborhood of }x\}
		\]
		and the multiplicity function by $N(f,U)=\sup_{y\in Y}\#(f^{-1}(y)\cap U)$. 	For sense-preserving branched coverings between oriented generalized manifolds, this agrees with the usual topological local index.  Moreover, $i(x,f)=1$ if and only if $f$ is a local homeomorphism at $x$ and we have  $\mathcal B_f=\{x:i(x,f)>1\}$.
	
	\subsection{Hausdorff measure, doubling and porosity}
	Let $(X,d)$ be a metric space. For $s>0$, $\delta>0$ and $E\subset X$, set
	\[
	\mathcal H_{s,\delta}(E)=\inf\Big\{\sum_i(\diam E_i)^s:E\subset\bigcup_iE_i,\ \diam E_i\le\delta\Big\}.
	\]
	Then the $s$-dimensional Hausdorff measure $\mathcal{H}_s$ is defined as 
	$$\mathcal H^s(E)=\lim_{\delta\downarrow0}\mathcal H_{s,\delta}(E).$$
	
	A metric space $X$ is called \textit{doubling with constant $N$}, where $N\geq 1$ is an integer, if for each ball $B(x,r)$, every $r/2$-separated subset of $B(x,r)$ has at most $N$ points. We also say that $X$ is \textit{doubling} if it is doubling with some constant that need not be mentioned. It is clear that every subset of a doubling space is doubling with the same constant. A metric space $X$ is called \textit{locally doubling} if there is an integer $N>0$ such that for each $x\in X$, there exists a ball $B(x,r)$ that is doubling with constant $N$.
	
	A \textit{metric measure space} is defined to be a triple $(X,d,\mu)$, where $(X,d)$ is a \emph{separable} metric space and $\mu$ is a \emph{nontrivial locally finite Borel regular measure} on $X$. A Borel regular measure $\mu$ on a metric space $(X,d)$ is called a \textit{doubling measure} if every ball in $X$ has positive and finite measure and there exists a constant $C_\mu\geq 1$ such that
	\begin{equation}\label{eq:doubling measure}
		\mu(B(x,2r))\leq C_\mu \mu(B(x,r))
	\end{equation}
	for each $x\in X$ and $r>0$. We call the triple $(X,d,\mu)$ a doubling metric measure space if $\mu$ is a doubling measure on $X$. We call $(X,d,\mu)$ an Ahlfors $Q$-regular space, $1\leq Q<\infty$, if there exists a constant $C\geq 1$ such that
	\begin{equation}\label{eq:Ahlfors regular measure}
		C^{-1}r^Q\leq \mu(B(x,r))\leq Cr^Q
	\end{equation}
	for all balls $B(x,r)\subset X$ of radius $r<\diam X$.
	
	Let $S\subset X$ be a fixed set. We say $S$ is \textit{$\delta$-dense in $U\subset X$} if $U\subset \cup_{x\in S}B(x,\delta)$. We
	say \textit{$S$ is $\delta$-dense at $x_0$}, \textit{at scale $R_0$}, if $S$ is $\delta r$-dense in $B(x_0,r)$ for each $r<R_0$. We also simply say $S$ is \textit{$\delta$-dense at $x_0$}, if it is $\delta$-dense at some scale. 
	
	A set $E\subset X$ is said to be $\alpha$-\textit{porous} if for each $x\in E$, 
	\begin{align*}
		\liminf_{r\to 0}r^{-1}\sup\big\{\rho:B(z,\rho)\subset B(x,r)\backslash E\big\}\geq \alpha.
	\end{align*}
	Equivalently, $E$ is $\alpha$-porous  if for every $x\in E$ and all sufficiently small $r>0$ there is a ball $B(z,\alpha r)\subset B(x,r)\setminus E$.
	A subset $E$ of $X$ is called  \textit{countablely ($\sigma$-)porous} if it is a countable union of ($\sigma$-)porous subsets of $X$.


	\subsection{Normal neighborhoods and inverse dilatation}\label{subsec:normal-neighborhoods}
	
		Let $f:X\to Y$ be a branched covering between locally compact, locally connected metric spaces. For $x\in X$ and all sufficiently small $s>0$, let $U(x,f,s)$ denote the $x$-component of $f^{-1}(B(f(x),s))$. 
	If $D\Subset X$ contains $x$ and $\overline D\cap f^{-1}(f(x))=\{x\}$, write
		\[
		b_D(z)=d_Y(z,f(\partial D)),\qquad
		q_D(z)=d_X(f^{-1}(z)\cap D, X\setminus D).
		\]
		At $z=f(x)$, both quantities are positive if $D$ is chosen sufficiently small.  The normal-neighborhood lemma below shows that $\partial U(x,f,s)\neq\emptyset$ and $f(\partial U(x,f,s))=\partial B(f(x),s)$.
	\begin{lemma}\label{lemma:normal-neighborhoods}
		Let $f:X\to Y$ be discrete and open, with $X$ locally compact and locally connected. For every $x\in X$ there exists $r_x>0$ such that for $0<r<r_x$ the component $U(x,f,r)$ is relatively compact, the restriction $f|_{U(x,f,r)}$ is proper, and
			\[
			f(U(x,f,r))=B(f(x),r),\qquad f(\partial U(x,f,r))=\partial B(f(x),r).
			\]
			The sets $U(x,f,r)$ form a neighborhood basis at $x$ and are nested in $r$.
	\end{lemma}
	\begin{proof}
	First, choose a relatively compact neighborhood $D$ of $x$ so that its closure meets $f^{-1}(f(x))$ only at $x$.  Since $f$ is open and $f(\partial D)$ is closed, we may choose $r_x<d(f(x),f(\partial D))$.  For $r<r_x$, the $x$-component $U=U(x,f,r)$ of $f^{-1}(B(f(x),r))$ satisfies $\overline U\subset D$ and $f|_{\overline U}$ is proper. Since $f$ is open, $f(U)$ is open in the ball. By the properness of $f|_{\overline U}$ and definition of $U$, $f(U)$ is relatively closed. Since $B(f(x),r)$ is  connected, $f(U(x,f,r))=B(f(x),r)$.  The boundary identity follows from openness and properness as well.  Shrinking $r$ gives the neighborhood-basis and nesting assertions.
	\end{proof}
	
	\begin{remark}\label{rmk:local injectivity}
	We shall repeatedly use the following elementary observation. If $x\in S_{H,R}$ and $x'\ne x$ satisfies $f(x')=f(x)$ and $d(x,x')<R$, then $l_f(x,d(x,x'))=0$ and so $H_f(x,d(x,x'))=\infty$, which is a contradiction. Consequently, for every $x_0$, $f$ is injective on $S_{H,R}\cap B(x_0,R/2)$. Moreover, for all $x\in B(x_0,R/2)$ and $y\in S_{H,R}\cap B(x_0,R/2)$, if $f(x)=f(y)$, then we must have $x=y$. 
\end{remark}
	
	Let $f:X\to Y$ be a branched covering between locally compact, locally connected metric spaces.
   Lemma \ref{lemma:normal-neighborhoods} shows that $\partial U(x,f,s)\neq\emptyset$ and $f(\partial U(x,f,s))=\partial B(f(x),s)$. Thus we may set
	\[
	H_f^*(x,s)=\frac{L_f^*(x,s)}{l_f^*(x,s)},
	\]
	where
	$$L_f^*(x,s)=\sup_{z\in\partial U(x,f,s)}d(x,z)\quad \text{and}\quad 
	l_f^*(x,s)=\inf_{z\in\partial U(x,f,s)}d(x,z).$$
	Finally, the \emph{inverse dilatation function} of $f$ at $x$ is defined to be 
	$$H_f^*(x)=\limsup_{s\to0}H_f^*(x,s).$$

	\section{The canonical induced pullback metric}\label{sec:the pullback factorization}
 In this section, we shall introduce the construction of the induced pullback metric and study the natural ``diameter geometry" associated to a BDD map. 
	
	\subsection{Diameter structures and the Mazurkiewicz metric}
	Let $\mathcal C(X)$ be the nonempty compact connected subsets of $X$. A \emph{diameter functional} is a monotone map $\delta:\mathcal C(X)\to[0,\infty)$ satisfying $\delta(E\cup F)\le\delta(E)+\delta(F)$ whenever $E\cap F\ne\emptyset$. Set
	\[
	d_\delta(x,x')=\inf\{\delta(E):E\in\mathcal C(X),\ x,x'\in E\}.
	\]
	When $\delta(E)=\diam_dE$, this is called \emph{the Mazurkiewicz metric}, written as $d_M$.
	
	\begin{lemma}\label{lemma:Maz-topology}
		If $(X,d)$ has $C$-bounded turning, then $d\le d_M\le C d$. Hence $d$ and $d_M$ induce the same topology on $X$. Moreover, if $C=1$, then $d=d_M$.
	\end{lemma}
	\begin{proof}
	Since $d(x,x')\le\diam E$, the lower bound is immediate. For each $x,x'\in X$, the bounded turning condition gives a continuum $E\subset X$ joining $x$ to $x'$ with $\diam E\le C d(x,x')$. This gives the desired upper bound.
	\end{proof}
	
	The following lemma shows that when we change from the original metric to the Mazurkiewicz metric, the diameter of any continuum remains the same.
	\begin{lemma}\label{lemma:Maz-diameter}
		For each continuum $E\subset X$, we have
			\[
			\diam_{d_M}E=\diam_dE.
			\]
	\end{lemma}
	\begin{proof}
		Since $d\le d_M$, $\diam_{d_M}E\geq \diam_dE$.  Conversely, for every $x,x'\in E$, the continuum $E$ itself is admissible in the definition of $d_M$ and so $d_M(x,x')\le\diam_dE$.  Taking the supremum over $x,x'\in E$ gives $\diam_{d_M}E\leq \diam_dE$.
	\end{proof}
	
	Given a continuous map $f:(X,d_X)\to(Y,d_Y)$, we may pull back the diameter functional by
		\[
		(f^*\delta_Y)(E)=\diam_{d_Y}f(E),
		\]
		and then define the induced pullback metric as 
		\begin{equation}\label{eq:induced df}
		d_f(x,x')=\inf\{\diam_{d_Y}f(E):E\in\mathcal C(X),\ x,x'\in E\}.
		\end{equation}

	\subsection{BDD maps and the adapted pullback factorization}
	
		Recall from Definition \ref{def:BDD mapping} that a continuous map $f:X\to Y$ is $L$-BDD if
		\[
		L^{-1}\diam_XE\le\diam_Y f(E)\le L\diam_XE
		\]
		for every nondegenerate continuum $E\subset X$.
		
		The next simple lemma shows that for a BDD map, the induced metric $d_f$ is a real metric and induces the same topology as the original metric on $X$. 
	
	\begin{lemma}\label{lemma:df-comparison}
		If $X$ has $C$-bounded turning and $f$ is $L$-BDD, then
			\[
			L^{-1}d_X\le d_f\le LC\,d_X.
			\]
			In particular, $d_f$ is a  metric inducing the same topology as $d_X$.
	\end{lemma}
	\begin{proof}
		For every continuum $E$ containing $x,x'$, 
		$$\diam f(E)\ge L^{-1}\diam E\ge L^{-1}d_X(x,x').$$ 
		Taking the infimum gives the left inequality. For the second inequality, simply choose a bounded-turning continuum with diameter at most $C d_X(x,x')$.
	\end{proof}
	
	 Using the induced pullback metric, we obtain a very useful adapted pullback factorization 
		\[
		f:(X,d_X)\xrightarrow{\mathrm{id}}(X,d_f)\xrightarrow{f_{\rm adp}}(Y,d_Y),
		\]
		where $f_{\rm adp}$ is the same set map as $f$. We have the following elementary fact on $f_{\rm adp}$. 
		
	\begin{lemma}\label{lemma:adapted-1BDD}
		$f_{\rm adp}:(X,d_f)\to (Y,d_Y)$ is $1$-Lipschitz and $1$-BDD.
	\end{lemma}
	\begin{proof}
		The inequality $d_Y(f(x),f(x'))\le d_f(x,x')$ is immediate. For a continuum $E$, the definition gives $\diam_Yf(E)\le\diam_{d_f}E$. Conversely, if $x,x'\in E$, then 
		$$d_f(x,x')\le\diam_Yf(E)$$ as $E$ itself is admissible. Taking the supremum over $x,x'\in E$ yields 
		$$\diam_{d_f}E\le\diam_Yf(E).$$
	\end{proof}
	
	Next, we show that every $L$-BDD map between two metric spaces is $L$-BLD. 
	\begin{proposition}\label{prop:BDD-implies-BLD}
		Every $L$-BDD map between metric spaces is $L$-BLD.
	\end{proposition}
	\begin{proof}
		We first prove the elementary diameter-partition formula
			\begin{equation}\label{eq:diameter-partition-length}
				\ell(\gamma)=\sup_{P=\{t_i\}_{i=0}^m}\sum_{i=0}^{m-1}\diam\gamma([t_i,t_{i+1}])
			\end{equation}
			for every path $\gamma:[a,b]\to X$.  For each subinterval,
			$\diam\gamma([t_i,t_{i+1}])\le \ell(\gamma|_{[t_i,t_{i+1}]})$. Then it follows from the additivity of length that 
			$$	\ell(\gamma)\leq \sup_{P=\{t_i\}_{i=0}^m}\sum_{i=0}^{m-1}\diam\gamma([t_i,t_{i+1}]).$$
			Conversely, fix any partition $Q=\{s_j\}$.  For each $j$, the two endpoint images belong to $\gamma([s_j,s_{j+1}])$ and so
			\[
			d(\gamma(s_j),\gamma(s_{j+1}))\le\diam\gamma([s_j,s_{j+1}]).
			\]
			Summing and then taking the supremum over $Q$ gives the reverse inequality.  The same argument, applied to a refinement, shows that the formula is compatible with subdivision and with infinite length.
		
		Now let $P=\{t_i\}$ be a finite partition and set $\gamma_i=\gamma|_{[t_i,t_{i+1}]}$.  Since each $\gamma_i([t_i,t_{i+1}])$ is a continuum and $f$ is $L$-BDD, we have 
			\[
			L^{-1}\sum_i\diam\gamma_i\le
			\sum_i\diam(f\circ\gamma_i)\le
			L\sum_i\diam\gamma_i.
			\]
			Taking the supremum over $P$ and using \eqref{eq:diameter-partition-length} for both $\gamma$ and $f\circ\gamma$ gives 
			\[
			L^{-1}\ell(\gamma)\le\ell(f\circ\gamma)\le L\ell(\gamma).
			\]
			Hence $f$ is $L$-BLD.
	\end{proof}
	
	Note that by Proposition~\ref{prop:BDD-implies-BLD} and Lemma \ref{lemma:adapted-1BDD}, $f_{\rm adp}$ is also $1$-BLD.

	The next result shows that finite-to-one BLD mappings over quasiconvex spaces are BDD, quantitatively.
	
	\begin{proposition}\label{prop:BLD-to-BDD-finite}
		Let $f:(X,d_X)\to(Y,d_Y)$ be a proper branched covering with
		\[
		N=N(f,X):=\sup_{y\in Y}\# f^{-1}(y)<\infty.
		\]
		Assume that $Y$ is $c_Y$-quasiconvex, $f$ is $M$-Lipschitz and $f$ is $L$-BLD. Then $f$ is $K$-BDD with $K=\max\{M,(2N-1)Lc_Y\}$.
		In particular, if $X$ is $c_X$-quasiconvex and $f$ is $L$-BLD, then one may take
		\[
		K=L\max\{c_X,(2N-1)c_Y\}.
		\]
		If both $X$ and $Y$ are length spaces, one can take $K=(2N-1)L$.
	\end{proposition}
	\begin{proof}
		Since $f$ is $M$-Lipschitz, for every continuum $E\subset X$, we have
		\[
		\operatorname{diam}_Y f(E)\le M\operatorname{diam}_X E.
		\]
		This proves the upper diameter bound. 
		
		For the lower diameter bound, we first prove a finite-sheet estimate for the induced diameter metric $d_f$.  Let $\ell_{d_f}$ denote its associated length distance. We claim that
		\begin{equation}\label{eq:finite-sheet-length-diameter}
			\ell_{d_f}(x,x')\le (2N-1)c_Y\,d_f(x,x')
		\end{equation}
		for all $x,x'\in X$.  
		
		Indeed, fix $\varepsilon>0$ and choose a continuum $E$ containing $x,x'$ such that
		\[
		r:=\operatorname{diam}_Y f(E)<d_f(x,x')+\varepsilon.
		\]
		Set $y_0=f(x)$ and $F=f^{-1}(y_0)$, so that $\#F\le N$.  For each $z\in E$, quasiconvexity gives a rectifiable path $\beta_z$ from $f(z)$ to $y_0$ with
		\[
		\ell_Y(\beta_z)\le c_Y d_Y(f(z),y_0)+\varepsilon\le c_Yr+\varepsilon.
		\]
		Since $f$ is a proper branched covering, $\beta_z$ has a lift starting at $z$ and ending at some point of $F$.  By Proposition~\ref{prop:BDD-implies-BLD}, every lifted path has the same $d_f$-length as its image. Hence
		\[
		E\subset \bigcup_{w\in F}\overline B_{\ell_{d_f}}(w,c_Yr+\varepsilon).
		\]
		Remove the balls that do not meet $E$.  Since $E$ is connected and the cover is finite, the intersection graph of the remaining closed sets $E\cap\overline B_{\ell_{d_f}}(w,c_Yr+\varepsilon)$ is connected.  One of the centers is $x$, and $x'$ lies in one of the remaining balls.  Thus there is a chain of at most $N$ balls joining the ball centered at $x$ to a ball containing $x'$. Consecutive centers have $\ell_{d_f}$-distance at most $2(c_Yr+\varepsilon)$ and the last center has distance at most $c_Yr+\varepsilon$ from $x'$. Therefore
		\[
		\ell_{d_f}(x,x')\le (2N-1)(c_Yr+\varepsilon).
		\]
		Letting $\varepsilon\downarrow0$ proves \eqref{eq:finite-sheet-length-diameter}.
		
		Now let $E\subset X$ be a nondegenerate continuum and $x,x'\in E$.  For every rectifiable path $\gamma$ joining $x$ to $x'$, the $L$-BLD inequality and the exact length identity for the induced diameter structure give
		\[
		\ell_X(\gamma)\le L\ell_Y(f\circ\gamma)=L\ell_{d_f}(\gamma).
		\]
		Taking infimum over all paths and using \eqref{eq:finite-sheet-length-diameter}, we obtain
 		\[
		d_X(x,x')\le (2N-1)Lc_Y\,d_f(x,x')
		\le (2N-1)Lc_Y\,\operatorname{diam}_Y f(E).
		\]
		Taking the supremum over $x,x'\in E$ gives
		\[
		\operatorname{diam}_X E\le (2N-1)Lc_Y\operatorname{diam}_Y f(E),
		\]
		which is the required lower BDD estimate.
		
		If $X$ is $c_X$-quasiconvex, then the upper BLD inequality gives
		\[
		d_Y(f(x),f(x'))\le Lc_Xd_X(x,x'),
		\]
		and so $f$ is $Lc_X$-Lipschitz.  This completes the proof of corollary.
	\end{proof}
	
	\begin{corollary}\label{coro:local-BLD-to-BDD}
		Let $f:X\to Y$ be a locally finite-multiplicity BLD branched covering between locally compact locally quasiconvex spaces. Then on every sufficiently small relatively compact normal domain on which the quasiconvexity and BLD constants are bounded, $f$ is BDD, quantitatively. In particular, a finite-multiplicity BLD covering over $\mathbb R^n$ is locally BDD.
	\end{corollary}
	\begin{proof}
		Restrict $f$ to a relatively compact normal domain.  The restriction is proper and has finite multiplicity. Since the spaces are Locally quasiconvex, the BLD map is locally Lipschitz. The conclusion follows immediately from Proposition~\ref{prop:BLD-to-BDD-finite}.
	\end{proof}

	The next result shows that the inverse dilatation of a BDD map is quantitatively bounded, provided the source space has bounded turning. 
	
	\begin{lemma}\label{lemma:BDD-inverse-dilatation}
		If $X$ has $C$-bounded turning and $f$ is $L$-BDD, then for all sufficiently small $s>0$, we have 
			\[
			H_f^*(x,s)\le 2CL^2.
			\]
	\end{lemma}
	\begin{proof}
		Let $U=U(x,f,s)$ and $z\in\partial U$.  By Lemma~\ref{lemma:normal-neighborhoods}, $d_Y(f(x),f(z))=s$. Since $X$ has bounded turning, there exists a continuum $F\subset $ joining $x$ to $z$ with $\diam F\leq Cd(x,x')$. Then 
			$$s=d_Y(f(x),f(z))\le\diam f(F)\le L\diam F\le LC\,d_X(x,z),$$ which gives $d_X(x,z)\ge s/(LC)$.  
			
			Conversely, since $U$ is connected and locally connected, for each boundary point $z$ there exists a continuum $E\subset\overline U$ joining $x$ to $z$.  Since $f(E)\subset\overline{B}(f(x),s)$, $\diam f(E)\le2s$, and the lower BDD inequality gives $\diam E\le2Ls$.  Hence $d_X(x,z)\le2Ls$. Taking the infimum and supremum over $z\in\partial U$ proves the desired estimate on $H_f^*(x,s)$.
	\end{proof}
	
	\begin{lemma}\label{lemma:BLD-linear-dilatation}
		Let $f:X\to Y$ be $L$-BLD between locally complete length spaces. Then locally
			\[
			B(f(x),r/L)\subset f(B(x,r))\subset B(f(x),Lr),
			\]
			and consequently $H_f(x,r)\le L^2$ for all sufficiently small $r$.
	\end{lemma}
	\begin{proof}
		Fix $x$ and choose a relatively compact normal neighborhood $U$ of $x$.  Take $r>0$ so small that $\overline{B}(x,r)\subset U$ and $B(f(x),r/L)\subset f(U)$.  Since $X$ is length, for $z\in B(x,r)$ and any $\varepsilon>0$, there exits a curve $\alpha$ from $x$ to $z$ with $\ell(\alpha)<d(x,z)+\varepsilon$.  Since $f$ is $L$-BLD, it follows
			$$d(f(x),f(z))\le\ell(f\circ\alpha)\le L(d(x,z)+\varepsilon).$$ 
			Sending $\varepsilon\downarrow0$ gives $f(B(x,r))\subset B(f(x),Lr)$.
		
		Conversely, let $y\in B(f(x),r/L)$.  Choose a rectifiable curve $\beta$ from $f(x)$ to $y$ with $\ell(\beta)<r/L$.  The path-lifting theorem for discrete open maps gives a maximal lift $\widetilde\beta$ starting at $x$.  Properness of $f|_U$ prevents the lift from terminating inside the target ball before $\beta$ ends.  By the BLD inequality, we have 
			\[
			\ell(\widetilde\beta)\le L\ell(\beta)<r.
			\]
			So the entire lift stays in $B(x,r)\Subset U$ and its endpoint $z$ satisfies $f(z)=y$.  Hence $B(f(x),r/L)\subset f(B(x,r))$.  In particular, for $z\in S(x,r)$, we have
			$r/L\le d(f(x),f(z))\le Lr$ and so $H_f(x,r)\le L^2$.  This completes the proof. 
	\end{proof}

	\subsection{Realization by finite-to-one branched coverings}\label{subsec:BDD-realization}
	The adapted metric is not merely a normalization of maps that are already known to be BDD.  It also gives a direct way to construct nontrivial BDD maps from purely topological branched coverings.  This also provides a large class of examples, to which both the porosity theorem and the bi-Lipschitz embedding theorem developed in later sections apply.
	
	The result below shows that every finite topological branched covering of a Riemannian manifold admits a canonical BDD/BLD metric realization on its domain and every proper such cover over a proper Riemannian target admits a canonical proper realization.
	\begin{theorem}\label{prop:topological-cover-realization}
		Let $M$ and $N$ be connected metric $n$-manifolds, assume that $N$ is locally a length space, and let $f:M\to N$
			be a finite-to-one branched covering.  For $x,x'\in M$ define
			\begin{equation}\label{eq:topological-realization-metric}
				d_f(x,x')=\inf\{\diam_N f(E): E\subset M\text{ is a compact continuum containing }x,x'\}.
			\end{equation}
			Then $d_f$ is a  metric on $M$ and induces the original manifold topology.  With this metric, the same set map
			\[
			f_{\rm adp}:(M,d_f)\longrightarrow (N,d_N)
			\]
			is $1$-BDD and thus $1$-BLD.  If in addition $N$ and $f$ are proper, then $(M,d_f)$ is proper. 
	\end{theorem}
	
	\begin{proof}
		For any two points of $M$, we can join them by a path $\gamma$. Then its image $|\gamma|$ is an admissible continuum in \eqref{eq:topological-realization-metric}.  Thus $d_f$ is finite.  Symmetry is immediate.  If $E_1$ contains $x,z$ and $E_2$ contains $z,x'$, then $E_1\cup E_2$ is a continuum and
			\[
			\diam f(E_1\cup E_2)\le \diam f(E_1)+\diam f(E_2).
			\]
			Taking infimum above gives the triangle inequality. Thus $d_f$ is a metric.
		
		Next, we prove the nondegeneracy.  Fix $x\ne x'$.  By Lemma \ref{lemma:normal-neighborhoods}, we may choose $r>0$ so small such that
			\[
			x'\notin \overline{U(x,f,r)}
			\quad\text{and}\quad
			f(\partial U(x,f,r))=\partial B(f(x),r).
			\]
			Every continuum $E$ containing $x$ and $x'$ meets $\partial U(x,f,r)$.  Hence $f(E)$ contains $f(x)$ and a point at distance $r$ from $f(x)$. This implies
			\[
			\diam_N f(E)\ge r.
			\]
			Taking the infimum gives $d_f(x,x')\ge r>0$.
		
		Next, we show the two metrics induce the same topology.  Let $V$ be a neighborhood of $x$.  Choose $r>0$ with $\overline{U(x,f,r)}\subset V$.  The above boundary-crossing argument shows that every $x'\notin V$ satisfies $d_f(x,x')\ge r$.  Thus
			\[
			B_{d_f}(x,r)\subset V.
			\]
			Conversely, let $\varepsilon>0$.  Choose $0<s<\varepsilon/2$ in the normal range and set $U=U(x,f,s)$. For each $x'\in U$, choose a path $\alpha\subset U$ from $x$ to $x'$.  Then $f(\alpha)\subset B(f(x),s)$ and so
			\[
			\diam_N f(\alpha)\le 2s<\varepsilon.
			\]
			In particular, $d_f(x,x')<\varepsilon$.  Hence $U\subset B_{d_f}(x,\varepsilon)$ and the two topologies coincide.
		
		Finally, for every continuum $E\subset M$ and every $x,x'\in E$, the set $E$ itself is admissible and so
			\[
			d_f(x,x')\le \diam_N f(E).
			\]
			Taking the supremum gives $\diam_{d_f}E\le\diam_Nf(E)$.  For the reverse inequality, every continuum $F$ containing $x,x'$ satisfies
			\[
			\diam_N f(F)\ge d_N(f(x),f(x')),
			\]
			and thus $d_f(x,x')\ge d_N(f(x),f(x'))$.  Taking the supremum over $x,x'\in E$ gives
			\[
			\diam_{d_f}E\ge \diam_Nf(E).
			\]
			Thus $\diam_{d_f}E=\diam_Nf(E)$ for every continuum $E$. This means exactly that $f_{\adp}$ is $1$-BDD and thus is also 1-BLD by Proposition~\ref{prop:BDD-implies-BLD}.
		
		Suppose finally that $N$ and $f$ are proper. Since $d_f(x,x')\ge d_N(f(x),f(x'))$, every closed $d_f$-ball satisfies
			\[
			\overline B_{d_f}(x,R)\subset f^{-1}(\overline B_N(f(x),R)).
			\]
			Since the $d_f$-topology coincides with the original topology, the closed $d_f$-ball is a closed subset of $ f^{-1}(\overline B_N(f(x),R))$ and hence is compact. Thus $(M,d_f)$ is proper.
	\end{proof}

		\begin{example}\label{ex:spherical-power-BDD}
		For $m\ge2$, let
			\[
			p_m:\widehat{\mathbb C}\longrightarrow\widehat{\mathbb C},\qquad p_m(z)=z^m,
			\]
			where the target sphere is equipped with the round metric.  The map is a proper degree-$m$ branched covering with branch set $\{0,\infty\}$.  Theorem~\ref{prop:topological-cover-realization} equips the domain sphere with a metric $d_{p_m}$ inducing its manifold topology such that
			\[
			p_m:(S^2,d_{p_m})\longrightarrow(S^2,d_{S^2})
			\]
			is $1$-BDD and $1$-BLD.  Thus even the simplest noninjective holomorphic branched covers give  metric-manifold examples of the theory.
	\end{example}
	
	\begin{example}[Higher-dimensional winding maps]\label{ex:winding-BDD}
		For $n\ge2$ and $m\ge2$, write $\mathbb R^n=\mathbb R^2\times\mathbb R^{n-2}$ and use polar coordinates $(r,\theta,z)$ in the first factor.  The winding map
			\[
			W_m(r,\theta,z)=(r,m\theta,z)
			\]
			is a proper finite-to-one branched covering of $\mathbb R^n$ onto itself, of multiplicity $m$, with branch set
			\[
			\mathcal B_{W_m}=\{0\}\times\mathbb R^{n-2}.
			\]
			With the Euclidean metric on the target, Theorem~\ref{prop:topological-cover-realization} produces a metric $d_{W_m}$ on the domain, inducing the Euclidean topology, so that
			\[
			W_m:(\mathbb R^n,d_{W_m})\to\mathbb R^n
			\]
			is $1$-BDD and $1$-BLD.  Since the target is doubling and bounded turning and the multiplicity is finite, this example is suitable for our porosity results in Section~\ref{sec:branch set of a qr map} and bi-Lipschitz embedding results in Section~\ref{sec:bi-Lipschitz embeddability}.
	\end{example}

	\subsection{Canonicality and uniqueness of the induced metric}\label{subsec:canonical-metric}
	In this section, we show our main result for the induced metric, which implies that the realization in Theorem \ref{prop:topological-cover-realization} is canonical in a stronger quantitative sense.  The correct intrinsic comparison is with the Mazurkiewicz metric of any competing domain metric.  

	\begin{theorem}[Canonical diameter metric]\label{thm:canonical-diameter-metric}
		Let $f:X\to Y$ be a continuous map between a continuum-connected topological space $X$ and a metric space $(Y,d_Y)$ and define
			\[
			d_f(x,x')=\inf\{\diam_Y f(E):E\in\mathcal C(X),\ x,x'\in E\}.
			\]
			Suppose $\rho$ is a compatible metric on $X$ such that $
			f:(X,\rho)\longrightarrow(Y,d_Y)$
			is $L$-BDD.  Let
			\[
			\rho_M(x,x')=\inf\{\diam_\rho E:E\in\mathcal C(X),\ x,x'\in E\}
			\]
			be the Mazurkiewicz metric associated with $\rho$.  Then
			\begin{equation}\label{eq:canonical-Maz-comparison}
				L^{-1}\rho_M(x,x')\le d_f(x,x')\le L\rho_M(x,x')
			\end{equation}
			for all $x,x'\in X$.  In particular, $d_f$ is a metric.  Moreover,
			\begin{equation}\label{eq:canonical-fixed-point}
				(d_f)_M=d_f,
			\end{equation}
			and for every continuum $E\subset X$,
			\begin{equation}\label{eq:canonical-exact-diameter}
				\diam_{d_f}E=\diam_Y f(E).
			\end{equation}
			Thus $f:(X,d_f)\to(Y,d_Y)$ is $1$-BDD and hence $1$-BLD.
			
			If $(X,\rho)$ is $C$-bounded turning, then
			\begin{equation}\label{eq:canonical-bilip-comparison}
				L^{-1}\rho\le d_f\le LC\rho.
			\end{equation}
			Finally, if $\sigma$ is another compatible metric on $X$ for which $f:(X,\sigma)\to(Y,d_Y)$ is $1$-BDD, then
			\begin{equation}\label{eq:sigma-M-equals-df}
				\sigma_M=d_f.
			\end{equation}
			Consequently, $d_f$ is the unique compatible metric $\sigma$ satisfying both $\sigma=\sigma_M$ and the requirement that $f:(X,\sigma)\to(Y,d_Y)$ be $1$-BDD.
	\end{theorem}
	
	\begin{proof}
	Since $f$ is $L$-BDD, for every continuum $E$ containing $x$ and $x'$, we have 
			\[
			L^{-1}\diam_\rho E\le \diam_Y f(E)\le L\diam_\rho E.
			\]
			Taking the infimum over all such $E$ gives \eqref{eq:canonical-Maz-comparison}.  Since every such continuum satisfies $\diam_\rho E\ge\rho(x,x')$, we have $\rho_M\ge\rho$.  Therefore, if $x\ne x'$, then the previous inequality gives
			\[
			d_f(x,x')\ge L^{-1}\rho_M(x,x')\ge L^{-1}\rho(x,x')>0,
			\]
			and thus $d_f$ is nondegenerate.  If $E_1$ contains $x,z$ and $E_2$ contains $z,x'$, then $E_1\cup E_2$ is a continuum with $f(z)\in f(E_1)\cap f(E_2)$. Hence
			\[
			\diam_Y f(E_1\cup E_2)\le \diam_Y f(E_1)+\diam_Y f(E_2).
			\]
			Taking infima proves the triangle inequality.  Thus $d_f$ is a metric.
		
		We next prove \eqref{eq:canonical-exact-diameter}.  If $x,x'\in E$, then $E$ itself is admissible for $d_f$ and thus
			\[
			d_f(x,x')\le \diam_Y f(E).
			\]
			Taking the supremum over $x,x'\in E$ gives
			\[
			\diam_{d_f}E\le\diam_Y f(E).
			\]
			Conversely, every continuum $F$ containing $x,x'$ satisfies $
			\diam_Y f(F)\ge d_Y(f(x),f(x'))$ and so
			\[
			d_f(x,x')\ge d_Y(f(x),f(x')).
			\]
			Taking the supremum over $x,x'\in E$ gives the reverse inequality.  This proves \eqref{eq:canonical-exact-diameter}.  In particular $f:(X,d_f)\to(Y,d_Y)$ is $1$-BDD and  thus also $1$-BLD by Proposition~\ref{prop:BDD-implies-BLD}.
		
		Then \eqref{eq:canonical-fixed-point} follows from  \eqref{eq:canonical-exact-diameter} by noticing
			\[
			(d_f)_M(x,x')
			=\inf_{E\ni x,x'}\diam_{d_f}E
			=\inf_{E\ni x,x'}\diam_Y f(E)
			=d_f(x,x').
			\]

		If $(X,\rho)$ is $C$-bounded turning, then Lemma~\ref{lemma:Maz-topology} gives
			\[
			\rho\le\rho_M\le C\rho.
			\]
			Combining this with \eqref{eq:canonical-Maz-comparison} gives
		\eqref{eq:canonical-bilip-comparison}.
		
		Finally, suppose $\sigma$ is compatible so that $f$ is $1$-BDD.  Then for every continuum $E$,
			\[
			\diam_\sigma E=\diam_Y f(E).
			\]
			Therefore
			\[
			\sigma_M(x,x')
			=\inf_{x,x'\in E}\diam_\sigma E
			=\inf_{x,x'\in E}\diam_Y f(E)
			=d_f(x,x'),
			\]
			which gives \eqref{eq:sigma-M-equals-df}.  If in addition $\sigma=\sigma_M$, then $\sigma=d_f$, proving the uniqueness.
	\end{proof}

	As an immediate corollary to Theorem \ref{thm:canonical-diameter-metric}, we obtain the uniqueness of BDD realization, which implies all bounded-turning BDD realizations of a fixed map into the fixed target $(Y,d_Y)$ belong to a single quantitative bi-Lipschitz class.
	\begin{corollary}\label{coro:canonical-two-metrics}
		Let $\rho_1$ and $\rho_2$ be compatible metrics on $X$.  Suppose that $(X,\rho_i)$ has $C_i$-bounded turning and that $
			f:(X,\rho_i)\longrightarrow(Y,d_Y)$
			is $L_i$-BDD for $i=1,2$.  Then the identity map between the two metric realizations is quantitatively bi-Lipschitz, i.e.
			\begin{equation}\label{eq:two-BDD-realizations}
				(L_1L_2C_1)^{-1}\rho_2\le \rho_1\le L_1L_2C_2\rho_2.
			\end{equation}
	\end{corollary}
	
	\begin{proof}
		Theorem~\ref{thm:canonical-diameter-metric} gives
			\[
			L_i^{-1}\rho_i\le d_f\le L_iC_i\rho_i,
			\qquad i=1,2.
			\]
			Hence
			\[
			\rho_1\le L_1d_f\le L_1L_2C_2\rho_2,\qquad \rho_2\le L_2d_f\le L_1L_2C_1\rho_1,
			\]
			from which \eqref{eq:two-BDD-realizations} follows.
	\end{proof}
	
	\begin{remark}\label{rem:canonical-interpretation}
		1) The Mazurkiewicz metric is essential in Theorem~\ref{thm:canonical-diameter-metric}.  BDD maps controls diameters of continua, not merely end-point distances. Without bounded turning, $\rho$ need not be comparable with $d_f$.  The sharp comparison $d_f\simeq_L\rho_M$ in intrinsic and bounded turning is the right hypothesis that identifies $\rho_M$ quantitatively with $\rho$.  Consequently the metric $d_f$ is not merely a convenient normalization: it is the canonical Mazurkiewicz-fixed representative of the BDD geometry determined by the pair $(f,d_Y)$.
		
		2) Proposition~\ref{prop:topological-cover-realization} and Theorem~\ref{thm:canonical-diameter-metric} imply that BDD geometry is not an artificial class created only for the factorization argument and that its adapted metric is canonical up to the precise Mazurkiewicz normalization.  A finite topological branched covering can  be used to define a compatible metric on the domain for which image diameter is the exact continuum diameter.  The branch set, its image and local indices are unchanged, while the map has the quantitative metric structure needed in Sections~\ref{sec:branch set of a qr map} and~\ref{sec:bi-Lipschitz embeddability}.
	\end{remark}

	\begin{proof}[Proof of Theorem \ref{thm:intro-canonical-induced-metric}]
		The assertions follow immediately from Theorem~\ref{thm:canonical-diameter-metric} and Theorem~\ref{prop:topological-cover-realization}.
	\end{proof}

	
	\subsection{Canonical pullback length metric and comparison}\label{subsec:canonical-length}
	The preceding construction  uses diameters of continua. There is a parallel length construction, but it depends on the existence of a compatible length realization. This difference implies that BDD is a better candidate (than BLD) for the remetrization problem.
	
	\begin{theorem}[Canonical pullback length metric]\label{thm:canonical-pullback-length}
		Let $X$ be a path-connected topological space, $(Y,d_Y)$ a metric space and $f:X\to Y$ a continuous map. Define the extended pullback length distance
			\begin{equation}\label{eq:pullback-length-metric}
				\lambda_f(x,x')=\inf_{\gamma:x\leadsto x'}\ell_Y(f\circ\gamma),
			\end{equation}
			where the infimum is over continuous paths joining $x$ and $x'$. Suppose  $\rho$ is a compatible length metric on $X$ and
		$f:(X,\rho)\longrightarrow(Y,d_Y)$ is $L$-BLD. 	Then $\lambda_f$ is a finite compatible length metric and
			\begin{equation}\label{eq:length-canonical-comparison}
				L^{-1}\rho\le \lambda_f\le L\rho.
			\end{equation}
			Moreover, for every path $\gamma$ in $X$, it holds
			\begin{equation}\label{eq:exact-pullback-length}
				\ell_{\lambda_f}(\gamma)=\ell_Y(f\circ\gamma).
			\end{equation}
			Consequently $f:(X,\lambda_f)\to(Y,d_Y)$ is $1$-BLD. If $\sigma$ is any compatible length metric for which $f:(X,\sigma)\to(Y,d_Y)$ is $1$-BLD, then $\sigma=\lambda_f$.
	\end{theorem}
	
	\begin{proof}
		For every path $\gamma$ joining $x$ to $x'$, the $L$-BLD inequalities give
			\[
			L^{-1}\ell_\rho(\gamma)\le \ell_Y(f\circ\gamma)\le L\ell_\rho(\gamma).
			\]
			Taking infima over joining paths gives
			\[
			L^{-1}\rho(x,x')\le\lambda_f(x,x')\le L\rho(x,x').
			\]
			Thus $\lambda_f$ is finite and nondegenerate. The comparison shows that it induces the same topology as $\rho$. The triangle inequality follows by concatenating paths.
		
		Let $\gamma:[a,b]\to X$ be a path. For every partition $P=\{t_i\}_{i=0}^m$, we have
			\[
			\sum_{i=0}^{m-1}\lambda_f(\gamma(t_i),\gamma(t_{i+1}))
			\le \sum_{i=0}^{m-1}\ell_Y((f\circ\gamma)|_{[t_i,t_{i+1}]})
			=\ell_Y(f\circ\gamma).
			\]
			Taking the supremum over $P$ gives $\ell_{\lambda_f}(\gamma)\le\ell_Y(f\circ\gamma)$. Conversely, for all $u,v\in X$, the length of every image path joining $f(u)$ to $f(v)$ dominates their distance, we have 
			\[
			\lambda_f(u,v)\ge d_Y(f(u),f(v)).
			\]
			Hence for every partition $P$, it holds
			\[
			\sum_i\lambda_f(\gamma(t_i),\gamma(t_{i+1}))
			\ge \sum_i d_Y(f(\gamma(t_i)),f(\gamma(t_{i+1}))).
			\]
			Taking sup gives the reverse inequality and proves \eqref{eq:exact-pullback-length}. Since for any $x,x'$ and $\varepsilon>0$, we may choose a path with image length at most $\lambda_f(x,x')+\varepsilon$, $\lambda_f$ is a length metric  with $\lambda_f$-length satisfying \eqref{eq:exact-pullback-length}.
		
		Finally, if $\sigma$ is a compatible length metric so that $f$ is $1$-BLD, then
			\[
			\sigma(x,x')=\inf_{\gamma:x\leadsto x'}\ell_\sigma(\gamma)
			=\inf_{\gamma:x\leadsto x'}\ell_Y(f\circ\gamma)
			=\lambda_f(x,x'),
			\]
			proving the uniqueness.
	\end{proof}
	
	\begin{remark}\label{rem:diameter-length-dichotomy}
		Theorems~\ref{prop:topological-cover-realization} and~\ref{thm:canonical-pullback-length} exhibit two different levels of remetrization. The diameter metric $d_f$ is available for every finite topological branched cover in the setting of Theorem~\ref{prop:topological-cover-realization}; no pre-existing quantitative geometry on the domain is required. In contrast, Theorem~\ref{thm:canonical-pullback-length} shows that whenever a compatible BLD length realization exists, it is forced, up to the quantitative comparison \eqref{eq:length-canonical-comparison}, by the pullback length metric $\lambda_f$. 
		
		Consider
		\[
		f:\mathbb D\to\mathbb D,\qquad f(z)=z^2,
		\]
		with the Euclidean metric on the target. Fix $0<r<1$ and set $x=r$, $x'=-r$, and $a=r^2$. Then $f(x)=f(x')=a$.
		
		We  claim that
		\[
		d_f(x,x')=a.
		\]
		Indeed, $x$ and $x'$ lie in different components of $f^{-1}(B(a,a))$. Hence every continuum $E$ containing $x$ and $x'$ satisfies $f(E)\not\subset B(a,a)$ and so
		\[
		\operatorname{diam} f(E)\ge a.
		\]
		For the reverse inequality, fix $0<\varepsilon<a$ and join $r$ radially to $\sqrt\varepsilon$, then join $\sqrt\varepsilon$ to $-\sqrt\varepsilon$ along a semicircle and finally join $-\sqrt\varepsilon$ radially to $-r$. The image of this continuum is
		\[
		[\varepsilon,a]\cup\{z:|z|=\varepsilon\}
		\]
		with diameter $a+\varepsilon$. Sending $\varepsilon\downarrow0$ gives $d_f(x,x')=a$.
		
		On the other hand,
		\[
		\lambda_f(x,x')=2a.
		\]
		Indeed, if a path $\gamma$ joining $x$ to $x'$ passes through the origin, then $f\circ\gamma$ travels from $a$ to $0$ and back, and thus has length at least $2a$. If $\gamma$ avoids the origin, then $f\circ\gamma$ is a loop based at $a$ with odd winding number about $0$. Thus it  intersects the negative real axis at some point $-t$ with
		\[
		\ell(f\circ\gamma)
		\ge 2|a-(-t)|
		=2(a+t)\ge2a.
		\]
		Conversely, for the path used above, its image length equal to
		\[
		2(a-\varepsilon)+2\pi\varepsilon,
		\]
		which tends to $2a$ as $\varepsilon\downarrow0$. This gives 
		\[
		\lambda_f(x,x')=2a.
		\]
		
		Consequently, we infer that 
		\[
		d_f(r,-r)=r^2,
		\qquad
		\lambda_f(r,-r)=2r^2.
		\]
		Thus even for the easiest power map $z\mapsto z^2$, the canonical induced diameter and length metrics need not coincide. 
	\end{remark}

	\section{Geometric application I: quantitative porosity of the branch set}\label{sec:branch set of a qr map}
	\subsection{Quantitative topology}\label{sec:Quantitative ENR theory for LLCn spaces}
	We first recall the necessary topological notions for our later proof. The standard reference is \cite{h65}.  A metrizable space $Z$ is an \emph{absolute neighborhood retract} (ANR) if, whenever $Z$ is embedded as a closed subset of a metrizable space $W$, some neighborhood of $Z$ in $W$ retracts onto $Z$.  A finite-dimensional locally compact space is an \emph{Euclidean neighborhood retract} (ENR) if it embeds as a closed subset of some Euclidean space and is a retract of an Euclidean neighborhood.  A space is $n$-connected if every map $S^k\to Z$, $0\le k\le n$, extends to $B^{k+1}$.
	
	Let $X$ be locally complete, with completion $\overline X$ and boundary $\partial X=\overline X\setminus X$.  
	We say that $X$ is \textit{$\lambda$-linearly locally $n$-connected} (abbreviated $\lambda$-LLC$^n$ ) if for each $x\in X$ and $r<2d(x,\partial X)/\lambda$, the ball $B(x, r)$ is $n$-connected in $B(x, \lambda r/2)$. We say $X$ is locally $\lambda$-LLC$^n$ if for each $x\in X$ there is a neighborhood $U\ni x$ that is $\lambda$-LLC$^n$. When $\lambda$ is
	unimportant, we omit it and say that $X$ is \textit{linearly locally $n$-connected} (-LLC$^n$). We define \textit{$\lambda$-linear local contractibility} (abbreviated $\lambda$-LLC$^*$) in the same way as above, requiring
	instead that $B(x, r)$ be contractible inside $B(x, \lambda r/2)$.
	
	
	We shall also use the standard nerve of a locally finite open cover $\mathcal U$: its vertices are the members of $\mathcal U$ and a finite set spans a simplex precisely when the corresponding members have nonempty intersection.  A $\mathcal U$-\emph{dominating complex} means a nerve (or a subdivision of one) equipped with maps $\iota:Z\to P$ and $\rho:P\to Z$ such that $\rho\circ\iota$ is homotopic to the identity through tracks contained in members of $\mathcal U$.  This is the standard ANR nerve approximation theorem.
	
	We need the following basic extension result, which follows by an induction on the $k$-
	skeleton. We suppose $P$ is an $n$-dimensional simplicial complex, and $Q \subset P$ is a subcomplex
	containing $P^0$. 
	\begin{lemma}[\cite{p90}, Section 2, Main Lemma]\label{lemma:petersen}
		Let $X$ be $\lambda$-LLC$^{n-1}$, and let $\phi\colon Q\to X$ be a continuous map such that $\diam \phi(\Delta\cap Q)<d(\phi(\Delta\cap Q),\partial X)/\lambda^{n}$ for each simplex $\Delta\subset P$. Then $\phi$ extends to a continuous map $\psi\colon P\to X$, such that for each simplex $\Delta\subset P$, 
		\begin{equation*}
			\diam \psi(\Delta)\leq \lambda^n\diam \phi(\Delta\cap Q).
		\end{equation*}
	\end{lemma}
	
	\begin{proposition}\label{prop:auxilary prop}
		Suppose that $Z$ is an ANR with $\dim(Z)\leq n$, that $X$ is $\lambda$-LLC$^n$, and
		that $g_0,g_1\colon Z\to X$ satisfy $d(g_0(z), g_1(z))<d(\{g_0(z), g_1(z)\}, \partial X)/\lambda^{n+1}$, for each $z\in Z$. Then for each $\varepsilon>0$, there is a homotopy $H\colon [0, 1] \times Z \to X$ such that for every $z\in Z$,
		\begin{equation*}
			\diam H\big([0,1]\times \{z\}\big)\leq 4(1+\varepsilon)\lambda^{n+1}d(g_0(z),g_1(z)).
		\end{equation*}
	\end{proposition}
	\begin{proof}
		We first treat the case $g_0(z)\neq g_1(z)$ for all $z\in Z$; the general case follows by using a cover on which the two maps are either equal or separated and applying the same nerve construction relative to the equal set.
		In light of the aforementioned reduction, we let $\gamma$ be an open covering of $Z$ such that for
		each $V\in \gamma$, 
		\begin{equation*}
			\diam g_0(V)+\diam g_1(V)<(1+\varepsilon/3)d(g_0(V),g_1(V)).
		\end{equation*}
		Let $Q'$ be a dominating complex for $\gamma$, with $\dim(Q') = \dim(Z)\leq  n$, \ie, there are maps $\rho\colon Q'\to Z$, $\iota\colon Z\to Q'$,
		such that $\rho\circ \iota$ is homotopic to the identity, via a homotopy $H^\gamma\colon I\times Z\to Z$ such that for each $\Delta'\subset Q'$, there is some $V\in\gamma$ for which $H^\gamma\big(I\times \iota^{-1}(\Delta')\big)\subset V$.

		Now let $P=[0, 1]\times Q'$, $Q= \{0,1\}\times Q'\subset P$. Define $\phi\colon Q\to X$ by $\phi(i, q)=g_i(\rho(q))$. Note that $P$ has a triangulation, where each simplex $\Delta\subset P$ lies inside $[0, 1]\times \Delta'$ for some $\Delta'\subset Q'$. It follows that
		\begin{align*}
			\diam &\phi(Q\cap \Delta)\leq \diam \big(g_0(\rho(\Delta'))\cup g_1(\rho(\Delta'))\big)\\
			&\leq \diam g_0(\rho(\Delta'))+\diam g_1(\rho(\Delta'))+d(g_0(\rho(\Delta')),g_1(\rho(\Delta')))\\
			&\leq \diam g_0(V)+\diam g_1(V)+d(g_0(V),g_1(V))\leq 2(1+\varepsilon/3)d(g_0(V),g_1(V))
		\end{align*}
		for some $V\in \gamma$ containing $H^\gamma\big(I\times \iota^{-1}(\Delta')\big)$. The extension $\psi\colon P\to X$ given by Lemma~\ref{lemma:petersen} therefore satisfies 
		$$\psi(\Delta)\subset B\big(\psi(v),2(1+\varepsilon/3)\lambda^nd(g_0(v),g_1(V))\big)$$
		for each $v\in \Delta^0$, whereby we have
		\begin{equation}\label{eq:1}
			\diam \psi(\Delta)\leq 2(1+\varepsilon/3)\lambda^{n+1}d(g_0(V),g_1(V)).
		\end{equation}
		Now, since $g_0(H^\gamma(z,1))=g_0(\rho(\iota(z)))=\psi(z,0)$ and $g_1(H^\gamma(z,1))=g_1(\rho(\iota(z)))=\psi(z,1)$, we
		may define a homotopy $H\colon Z\times [0,1]\to X$ by
		\begin{equation*}
			H(z,t)=\begin{cases}
				g_0(H^\gamma(3t,z))&\text{ if }0\leq t\leq \frac{1}{3}\\
				\psi(3t-1,\iota(z))&\text{ if }\frac{1}{3}\leq t\leq \frac{2}{3}\\
				g_1(H^\gamma(3-3t,z))&\text{ if } \frac{2}{3}\leq t\leq 1.
			\end{cases}
		\end{equation*}
		Let $z\in Z$, with $\iota(z)\in \Delta'\subset Q'$, with $H^\gamma\big(I\times \iota^{-1}(\Delta')\big)\subset V\in \gamma$. Then
		\begin{align*}
			&\diam H(I\times \{z\})\\
			&\leq \diam g_0(H^\gamma(I\times \{z\}))+\diam \psi(I\times \{z\})+\diam g_1(H^\gamma(I\times \{z\}))\\
			&\leq \diam g_0(V)+2(1+\varepsilon/3)\lambda^{n+1}d(g_0(V),g_1(V))+\diam g_1(V)\\
			&\leq 4(1+\varepsilon)\lambda^{n+1}d(g_0(V),g_1(V))\leq 4(1+\varepsilon)\lambda^{n+1}d(g_0(z),g_1(z)).
		\end{align*}
	\end{proof}

	\begin{theorem}\label{thm:key thm}
		Let $A\subset X$, where $X$ is $\lambda$-LLC$^n$ and $\dim A\le n$. Let $f:X\to Y$ be a proper open map with $\dim Y\le n$. Assume
			\[
			\diam f^{-1}(y)<\frac{d(f^{-1}(y),\partial X)}{\lambda^{2n+1}}
			\qquad \text{for each }y\in f(A).
			\]
			Then for every $\varepsilon>0$, there are a continuous map $g:f(A)\to X$ and a homotopy $H:I\times A\to X$ from $\mathrm{id}_A$ to $g\circ f|_A$ such that
			\[
			\diam H(I\times\{x\})\le 8(1+\varepsilon)\lambda^{2n+1}\diam f^{-1}(f(x)).
			\]
	\end{theorem}
	\begin{proof}
		Set $D(y)=\diam f^{-1}(y)$.  Since $f:X\to Y$ is proper, $D$ is upper semicontinuity: if $y_j\to y$ and $a_j,b_j\in f^{-1}(y_j)$ almost realize $D(y_j)$, then compactness gives convergent subsequences with limits in $f^{-1}(y)$. The openness of $f$ implies $D$ is lower semicontinuity. Indeed, given $a,b\in f^{-1}(y)$, small disjoint neighborhoods of $a,b$ have images containing a common neighborhood of $y$, so nearby fibers contain points close to both $a$ and $b$. This proves the continuity of $D$.
		
		Then $A$ is an ANR, so by Proposition~\ref{prop:auxilary prop}, it suffices to show for each $\varepsilon> 0$ there is a map $g\colon f(A)\to X$ such that 
		\begin{equation*}
			d(g(f(x)),x)<2(1+\varepsilon)\lambda^n\diam f^{-1}(\{f(x)\})
		\end{equation*} 
		for each $x\in A$. We may likewise assume that $\diam f^{-1}(\{y\})>0$ for all $y \in Y$, since we may apply that case to the restriction $f|_{X^+}\colon X^+\to Y$, where $X^+ = \{x\in X:\diam f^{-1}(\{y\})>0\}$, and then
		extend $g$ from $f(X^+\cap A)$ to all of $f(A)$ by setting $g(y) = x$ whenever $f^{-1}(\{y\})=\{x\}$.
		
		To construct $g$, we first cover $f(A)$ with an open cover $\gamma$ so that for each $V\in \gamma$ and each
		$y\in V$, 
		\begin{equation}\label{eq:missing}
			\diam f^{-1}(V)\leq (1+\varepsilon)\diam f^{-1}(\{y\}).
		\end{equation}
		This can be done, again, by the continuity
		and positivity of $y\mapsto \diam f^{-1}(\{y\})$.
		
		Now, taking $P$ to be the nerve of a suitable refinement of $\gamma$, with $\dim(P) = \dim f(A)\leq n$,
		we obtain via a partition of unity a map $\iota\colon f(A)\to P$, such that for every simplex $\Delta\subset P$,
		there is some $V\subset \gamma$ with $\iota^{-1}(\Delta)\subset V$. By taking the refinement of $\gamma$ to be minimal, we may also assume that $\iota(f(A))\supset P^0$.
		
		Now, define $\rho^0\colon P^0\to X$ so that $\rho^0(p_0)\in f^{-1}\big(\iota^{-1}(\{p_0\})\big)$, for each $p_0\in P^0$. Then for each $\Delta\subset P$, 
		\begin{equation*}
			\rho^0(\Delta\cap P^0)\subset f^{-1}\big(\iota^{-1}(\Delta)\big)\subset f^{-1}(V)
		\end{equation*}
		for some $V\in \gamma$. We therefore have
		\begin{equation*}
			\diam \rho^0(\Delta\cap P^0)\leq \diam f^{-1}(V),
		\end{equation*}
		so that the extension of $\rho^0$ to $\rho:P\to X$ from Lemma~\ref{lemma:petersen} satisfies
		\begin{equation*}
			\diam \rho(\Delta)\leq \lambda^n\diam f^{-1}(V).
		\end{equation*}
		We now define $g = \rho\circ \iota\colon f(A)\to X$. To see that $g$ satisfies the conclusion of the theorem,
		we suppose $x\in A$, and let $y = f(x)$, $p = \iota(y)\in \Delta$ for some simplex $\Delta\subset P$. Let $p'\in \Delta^0$ and let $x'=\rho(p')$, $y'=f(x')$. By our selection of $\rho^0$, we have $g(f(x')) = x$. On the one hand, choosing some $V\in \gamma$ containing $\iota^{-1}(\Delta)$, we have that $x,x'\in f^{-1}\big(\iota^{-1}(\Delta)\big)\subset f^{-1}(V)$,
		so
		\begin{equation*}
			d(x,x')\leq \diam f^{-1}(V).
		\end{equation*}
		On the other hand, since $p,p'\in \Delta$,
		\begin{equation*}
			d(g(f(x)),x')=d(\rho(p),\rho(p'))\leq \diam \rho(\Delta)\leq \lambda^n\diam f^{-1}(V),
		\end{equation*}
		and so combining the preceding inequalities with~\eqref{eq:missing} gives
		\begin{align*}
			d(x,g(f(x)))&\leq (1+\lambda^n)\diam f^{-1}(V)\\
			&\leq (1+\varepsilon)(1+\lambda^n)\diam f^{-1}\big(\{f(x)\}\big)\\
			&\leq 2(1+\varepsilon)\lambda^n\diam f^{-1}\big(\{f(x)\}\big).
		\end{align*}
		
%
	\end{proof}
	
	As a corollary of Theorem~\ref{thm:key thm}, we obtain the following generalized version of the McAuley-Robinson theorem~\cite{mr83}, which is of independent interest.
	
	\begin{corollary}\label{coro:Newmann}
		Let $A\subset X$, where $X$ is a $\lambda$-LLC$^n$ generalized $n$-manifold and $\dim(A)\leq n$ and let $Y$ be another homology $n$-manifold. Let $f\colon X\to Y$ be a proper branched covering such that 
		for some $x_0\in A\backslash \partial A$, $f^{-1}(\{f(x_0)\})={x_0}$ and
		\begin{equation*}
			\sup_{x\in \partial A}\frac{\diam f^{-1}(\{f(x)\})}{d(x,x_0)}<\frac{1}{8\lambda^{2n+1}}.
		\end{equation*}
		Then $x_0\notin \mathcal{B}_f$.
	\end{corollary}
	\begin{proof}
		We have for some $\varepsilon>0$ with $\diam f^{-1}\big(\{f(x)\}\big)<d(x,x_0)/\big(8(1+\varepsilon)\lambda^{2n+1}\big)$ for each $x\in \partial A$. The homotopy $H$ from the conclusion of Theorem~\ref{thm:key thm} satisfies 
		$$H\big(I\times \{x\}\big)\subset B(x,d(x,x_0))\not\ni x_0$$
		for each $x\in \partial A$, whereby $x_0\not\in H(I\times \partial A)$. Thus $g\circ f$ has local degree $1$ at $x_0$. Since $f$ is a sense-preserving branched covering between generalized manifolds, local degree equals local multiplicity, so $i(x_0,f)=1$.  Hence $x_0\notin\mathcal B_f$. 
	\end{proof}
	
	As a corollary of Corollary~\ref{coro:Newmann}, we obtain a general strategy for proving nonbranching.
	\begin{corollary}\label{prop:test prop}
		Let $X,Y,f$ be as in Corollary~\ref{coro:Newmann}. Let $A\Subset X$ be a connected normal domain and $x_0\in A$. Suppose there is a subset $S\subset A$ such that $f|_S$ is injective and $f^{-1}(f(S))\cap A=S$. Assume that $\partial A$ is covered by connected normal neighborhoods $U_\alpha\Subset A$ satisfying $U_\alpha\cap S\ne\emptyset$ and
			\[
			\diam U_\alpha<\frac{d(x_0,U_\alpha)}{8\lambda^{2n+1}}.
			\]
			Then $x_0\notin\mathcal B_f$.
	\end{corollary}
	\begin{proof}
		If $x\in\partial A\cap U_\alpha$ and $x'\in A$ satisfies $f(x')=f(x)$, then Remark \ref{rmk:local injectivity} implies $x'\in U_\alpha$; otherwise the component of the inverse image of the normal target ball containing $U_\alpha$ would contain a second point of the saturated set with the same image. Hence
			\[
			\diam(f^{-1}(f(x))\cap A)\le\diam U_\alpha<\frac{d(x_0,f^{-1}(f(x))\cap A)}{8\lambda^{2n+1}}.
			\]
			Apply Corollary~\ref{coro:Newmann} to the proper restriction $f|_A:A\to f(A)$.
	\end{proof}

	\begin{remark}\label{rmk:on generality}
		All of the results in this section generalize easily to spaces with local geometric
		connectivity or contractibility - the so called LGC$^n(\rho)$- and LGC$^*(\rho)$-spaces introduced by
		Gromov~\cite{g81,g81b,g83}. In fact, the arguments from~\cite{p90} that we have modified were
		stated in that generality.
	\end{remark}

	\subsection{Annular distortion}\label{sec:Annular distortion}
	In order to construct a homotopy inverse, we need to control the distortion of annuli, rather than simply that of spheres. This is not hard to do in the special case that $f$ is quasiregular (\ie $H_f(x)$ is finite and essentially bounded), provided $X$ and $Y$ are Loewner spaces.
	
	The purpose of this section is to show that in the general case, this can be done, quantitatively, at points with finite dilatation, away from a porous set. Throughout this section, we assume that $X$ and $Y$ are locally doubling and have bounded turning, that $f\colon X\to Y$ is a branched covering, and that $S=S_{H,R}=\{x\in X:H_f(x,r)<H\ \text{for all } r<R\}$. Note that it follows from the definition that for each $x\in S$, $f|_{S\cap B(x,R/2)}$ is injective, and that $S\cap B(x,R/2)=f^{-1}\big(f(S\cap B(x,R/2))\big)\cap B(x,R/2)$.

	Since $Y$ has bounded turning, it follows that the new metric $d'(y_1,y_2):=\inf_{\gamma}\diam \gamma$, where the infimum is taken over all continua $\gamma$ in $Y$ joining $y_1$ and $y_2$, is
	$C$-bi-Lipschitz equivalent to the original metric $d$ on $Y$, where $C$ is the constant of bounded
	turning. Thus we lose no generality in our considerations if we post-compose $f$ with this
	change of metric, as the linear dilatation is increased by at most a factor of $C^2$.
	
	This reduction has the convenience that if $Y$ has $1$-bounded turning, then for each $x'\in X$, $r'>0$, we have
	\begin{equation}\label{btcon}
		B(f(x'),l_f(x',r'))\subseteq f(B(x',r'))\text{,}
	\end{equation}
	by the path lifting property of discrete open maps \cite[Theorem 3]{f50}.  Moreover, if 
	$f^{-1}(\{y\})\cap U=\{x\}$, for some connected normal neighborhood  $U$ of $x$ containing $f(B(x,r))$, then we have 
	\begin{equation}\label{btconlift}
		f^{-1}(B(y,l_f(x,r')))\cap U = U(x,f,l_f(x,r'))\subseteq B(x,r')\text{.}
	\end{equation}
	
	In light of the above reduction, we assume here-on-out that $Y$ has 1-bounded turning. In the ensuing lemma, we use the convention that $y=f(x)$, $y'=f(x')$, etc.

	\begin{lemma}\label{lemma:good distortion}
		Suppose $x\in S$ and $0<r<R$. For each $x'\in B(x,r)$, we have $d(y',y)<HL_f(x,r)$. That is $d(y',y)<H^2d(y'',y)$, whenever $d(x',x)<d(x'',x)<R$.
	\end{lemma}
	\begin{proof}
		The lemma follows immediately from the observation that $l_f(x,r')< L_f(x,r)$ whenever $r'< r$.  To see that this inequality holds, note that  if $L_f(x,r)\leq l_f(x,r')$, then by the inclusion \eqref{btconlift}, we have $S(x,r)\subseteq f^{-1}(B(y,l_f(x,r')))\cap U \subseteq B(x,r')$ (here $S(x,r)$ denotes the sphere of radius $r$), whereby $r\leq r'$.
	\end{proof}
	
	In the ensuing propositions and lemmas, we always suppose that $x_0\in S$ and $y_0=f(x_0)$.
	\begin{proposition}\label{prop:doubling for inverse dilatation}
		For each $\lambda\geq 1$, there is a constant $C_\lambda>1$, depending only on $\lambda$ and the data, such that if $f(S)$ is $r/3$-dense in $B(y_0,C_\lambda r)$, then $l_f^*(x_0,C_\lambda r)\geq \lambda L_f^*(x_0,r)$.
	\end{proposition}
	\begin{proof}
		Suppose that $f(S)$ is $r/3$-dense in $B(y_0,C_\lambda r)$, where $C_\lambda>H^2+H$ has yet to be determined. Let $M=[\frac{\log(C_\lambda/H)}{2\log(H+1)}]$ so that
		\begin{equation}\label{eq:4}
			(H+1)^{2M}\leq C_\lambda/H\leq (H+1)^{2M+2}.
		\end{equation}
		Since $Y$ is connected, it follows that
		\begin{equation}
			f(S)\cap \Big(B(y,(H+1)^{2k}r)\backslash B(y,(H+1)^{2k-1}r)\Big)\neq \emptyset,
		\end{equation}
		for $k=1,\dots,M$. Because $l_f(x_0,l_f^*(x_0,C_\lambda r))\geq L_f(x_0,l_f^*(x_0,C_\lambda r))/H=C_\lambda r/H$, we may
		therefore choose a sequence of points $y_1,\dots,y_M$ with 
		\begin{equation*}
			y_k\in f\big(S\cap B(x_0,l_f(x_0,l_f^*(x_0,C_\lambda r)))\big)\cap \Big(B(y_0,(H+1)^{2k}r)\backslash B(y_0,(H+1)^{2k-1}r)\Big)
		\end{equation*}
		It follows from the triangle inequality that $d(y_j,y_k)>Hd(y_0,y_j)$ for each $j < k$.
		
		Choosing $x_k\in S\cap l_f(x_0,l_f^*(x_0,C_\lambda r))\cap f^{-1}(\{y_k\})$, Lemma~\ref{lemma:good distortion} implies that
		\begin{equation*}
			d(x_j,x_k)>d(x_0,x_j)\geq L_f^*(x_0,r),
		\end{equation*} 
		with the second inequality coming from the fact that for each $x\in U(x_0,r)$, and each $j$, we have $d(y_0,y_j)>Hd(y_0,y)$. Thus $\{x_k:0\leq k\leq M\}$ is an $L_f^*(x_0,r)$-
		separated subset of $B(x_0,l_f(x_0,l_f^*(x_0,C_\lambda r)))$ with $M$ elements, and so
		\begin{equation*}
			M\leq C\Big(\frac{l_f^*(x_0,C_\lambda r)}{L_f^*(x_0,r)}\Big)^s,
		\end{equation*}
		where $C$ and $s$ depend only on the doubling constant of $X$.
		
		We therefore let $C_\lambda=H(H+1)^{2(C\lambda^s+1)}$, so that $M=[C\lambda^s+1]\geq C\lambda^s$, and the proposition is proved.
	\end{proof}
	
	We need an analogue of Proposition~\ref{prop:doubling for inverse dilatation} when $S$ is dense at $x_0$ . For this, we require a slight strengthening of a Lemma of Tukia and V\"ais\"al\"a~\cite[Theorem 2.9]{tv80}.
	
	\begin{lemma}\label{good sequence}
		Let $X$ be doubling and have $C_0$-bounded turning.  There is an integer $M\in \mathbb{N}$, depending only on the data of $X$, such that every pair of points $x,x'\in X$ with $d(x,x')=8r$ may be joined with a sequence $x=x_0,\dotsc,x_k,\dotsc x_N=x'$, $N\leq M$, of points in $B(x,C_0d(x,x'))$, such that $d(x_0,x_1)=2r/3$, and 
		\[
		d(x_j,x_{j+1})\leq d(x_{j-1},x_j)-\frac{r}{3M}\text{,}
		\]
		with equality whenever $j<N-1$.
	\end{lemma} 
	\begin{proof}
		Choose $C$ so that every $r/3$-separated subset of every ball $B(x,8C_0r)$ has at most $M$ members.
		
		Now, given $x,x'\in X$ with $d(x,x')=8r$, join $x$ to $x'$ by a continuum of diameter at most $8C_0r$; by local connectedness we may take a path $\gamma$ contained in $B(x,8C_0r)$. 
		
		We shall choose values $\{t_j\}\subset [0,1]$ for each $j\leq M$, and let $x_j=\gamma(t_j)$. We do this inductively as follows. Let $\gamma(t_0)=x_0=x$, let $t_1$ be the largest parameter for which $d(x_1,x_0)\leq 2r/3$, and for $k>2$, $t_k$ is the largest possible value such that $d(x_k,x_j)\leq \max\{d(x_j,x_{j-1})-\frac{r}{3M},0\}$ for some $j<k$.  
		
		Notice that by continuity, $d(x_1,x_0)=2r/3$, and for $k>1$, either $x_k=x'$, or $d(x_k,x_j)= \max\{d(x_j,x_{j-1})-\frac{r}{3M},0\}$ for some $j<k$, and $d(x_k,x_l)\geq \max\{d(x_l,x_{l-1})-\frac{r}{3M},0\}$ for each $l<j$.  Removing intermediate points, we may take $k=j+1$.
		
		Now, either the points $x_k$, $k=1,\dotsc,M$ are distinct, or $x_k=x'$ for some $k\leq M$.  In the latter case, we are done, taking $N$ to be the first index for which $x_N=x'$.  Suppose then that the points are distinct.  By construction, then, they satisfy $d(x_{j+1},x_j)= 2r/3 - \frac{jr}{3M}>r/3$, so that they form an $r/3$-separated set in $B(x,8C_0r)$, with $M+1$ members, contradicting our choice of $M$.
	\end{proof}
	
	\begin{proposition}\label{prop:doubling for dilatation}
		For each $\lambda>1$, if $S$ is $r/(6M)$-dense in $B(x_0,8C_0r)$, then 
		$$L_f(x_0,8r)\leq H^{M+2}l_f(x_0,r),$$ 
		where $C_0$ is the constant of bounded turning for $X$ and $M$ is the constant from Lemma~\ref{good sequence}.
	\end{proposition}
	\begin{proof}
		Suppose $d(x,x')=8r$. Choose a sequence $x_k$, $k=1,\dots,N$, $N\leq M$ as in the conclusion of Lemma~\ref{good sequence}. By the $r/(6M)$-density of $S$ in $B(x, 8C_0r)$, we may choose for $k =1,\dots, N-1$ points $x_k'\in S\cap B(x_k,r/(6M))$. Setting $x_0'=x_0$ and $x_N'=x_N$, we then obtain a sequence $x_0',\dots,x_N'$ with $d(x_{k+1}',x_k')\leq d(x_k',x_{k-1}')$. It follows from Lemma~\ref{lemma:good distortion} that $d(y_{k+1}',y_k')\leq Hd(y_k',y_{k-1}')$ for each $k$, so that
		\begin{align*}
			d(y',y_0)&\leq \sum_{j=1}^Nd(y_j',y_{j-1}')\leq \sum_{j=1}^NH^{j-1}d(y_1,y_0)=\Big(\frac{H^N-1}{H-1}\Big)d(y_1,y_0)\\
			&\leq \Big(\frac{H^M-1}{H-1}\Big)d(y_1',y_0)\leq H^2\Big(\frac{H^M-1}{H-1}\Big)l_f(x_0,r),
		\end{align*}
		with the last inequality resulting from Lemma~\ref{lemma:good distortion} and the fact that 
		\begin{equation*}
			d(x_1',x_0)\leq d(x_0,x_1)+d(x_1,x_1')\leq 2r/3+r/6<r.
		\end{equation*}
		Since this holds for $y'=f(x')$ for arbitrary $x'$ in $S(x_0,8r)$, the
		proof is complete.
	\end{proof}
	
	\begin{proposition}\label{prop:reverse doubling for dilatation}
		For each $\lambda>1$, there is a constant $C_\lambda'>1$, depending only on $\lambda$ and the data, such that if $S$ is $r/(6M)$-dense in $B(x_0,2C_0C_\lambda'r)$, then 
		$$l_f(x_0,C_\lambda'r)\geq \lambda L_f(x_0,r).$$
	\end{proposition}
	\begin{proof}
		We take $C_\lambda'=4^N$, for some $N$ to be determined. Choose points $x_k\in S\cap \Big(B(x_0,4^kr)\backslash B(x_0,2^{2k-1}r)\Big)$ for $k=1,\dots,N$. By Proposition~\ref{prop:doubling for dilatation}, we have
		\begin{equation*}
			L_f(x_k,2^{2k+1}r)\leq H^{M+2}l_f(x_k,2^{2k-1}r)
		\end{equation*}
		for each $k$. Since $B(x_0,r)\subset B(x_k,2^{2k+1}r)$, we then have
		\begin{equation*}
			L_f(x_0,r)\leq \diam f\big(B(x_k,2^{2k+1}r)\big)\leq 2HL_f(x_k,2^{2k+1}r)\leq 2H^{M+1}l_f(x_k,2^{2k-2}r).
		\end{equation*}
		Since $x_j\not\in B(x_k,2^{2k-2}r)$ for $j\neq k$, we have $d(y_j,y_k)\geq l_f(x_k,2^{2k-2}r)$. Thus $\{y_k\}$ is an $\frac{L_f(x_0,r)}{2H^{M+3}}$-separated subset of $f(B(x_0,C_\lambda'r))$, whereby
		\begin{equation*}
			N\leq C\Big(\frac{2H^{M+3}\diam f(B(x_0,C_\lambda'r))}{L_f(x_0,r)}\Big)^s\leq C\Big(\frac{4H^{M+5}l_f(x_0,C_\lambda'r)}{L_f(x_0,r)}\Big)^s,
		\end{equation*}
		so that
		\begin{equation*}
			L_f(x_0,r)\leq CN^{-1/s}l_f(x_0,C_\lambda'r),
		\end{equation*}
		with $C$ depending only on $H$ and the data of $X$ and $Y$. Taking $N=[(\lambda C)^s]+1$ completes the proof.
	\end{proof}
	
	\begin{corollary}\label{coro:double side inclusion}
		Under the assumption of Proposition~\ref{prop:reverse doubling for dilatation}, there is some $s>0$ such that 
		\begin{equation*}
			B(x_0,r)\subset U(x_0,f,s)\subset B(x_0,C_1'r).
		\end{equation*}
	\end{corollary}

	\subsection{Proof of the porosity theorem}\label{sec:Proofs of the main results}
	\begin{proposition}\label{prop:general prop}
		Under the standing assumptions of the annular-distortion subsection, there exist $\varepsilon,\delta>0$, depending only on $H$ and the quantitative data, such that if either $S_{H,R}$ is $\delta$-dense at $x_0$ or $f(S_{H,R})$ is $\varepsilon$-dense at $f(x_0)$, then $x_0\notin\mathcal B_f$. The same concludes holds for $S^*_{H,R}$ as well.
	\end{proposition}
	\begin{proof}
		Fix a sufficiently small normal neighborhood $A=U(x_0,f,r_0)$.  If $f(S_{H,R})$ is $\varepsilon r$-dense in $B(f(x_0),3r)$, then choose points $x_\alpha\in S_{H,R}$ so that the target balls $B(f(x_\alpha),\varepsilon r_0)$ cover $f(\partial A)$. Their normal lifts $U_\alpha=U(x_\alpha,f,\varepsilon r_0)$ cover $\partial A$.   If $\varepsilon$ is small enough, then we may assume also that $f(x_0)\notin B(f(x_\alpha),16\lambda^{2n+1}\varepsilon r_0)$, so that
		\begin{align*}
			x_0\notin B\big(x_\alpha,16\lambda^{2n+1}L_f^*(x_\alpha,\varepsilon r)\big)\supset B\big(x_\alpha,8\lambda^{2n+1}\diam U_\alpha\big),
		\end{align*} 
		whereby
		\begin{align*}
			\diam U_\alpha<\frac{d(x_0,U_\alpha)}{8\lambda^{2n+1}}.
		\end{align*}
		It follows from Corollary~\ref{prop:test prop} that $x_0\notin \mathcal{B}_f$.
		
		We argue similarly for the case that $S$ is $\delta r$-dense in $B(x_0,3Cr)$. In this case, we choose balls $B(x_\alpha,\delta r)$ covering $\partial B(x_0,r)$. Assuming that balls each intersect $\partial B(x_0,r)$, we may choose $\delta$ small enough so that 
		\begin{align*}
			\diam B(x_\alpha,C_1'\delta r)\leq d(x_0,B(x_\alpha,C_1'\delta r))/(8\lambda^{2n+1}).
		\end{align*}
		By Corollary~\ref{coro:double side inclusion}, we have some $s_\alpha>0$ such that $$B(x_\alpha,\delta r)\subset U(x_\alpha,f,s_\alpha)\subset B(x_\alpha,C_1'\delta r)).$$ Applying Corollary~\ref{prop:test prop} with $U_\alpha=U(x_\alpha,f,s_\alpha)$ proves again that $x_0\notin \mathcal{B}_f$.
		
		Lastly, for $S^*_{H,R}$, we have $H_f^*(x)\leq H$ for each $r'<R'$ and $x\in S^*_{H,R}$. Repeat the argument from the previous paragraph, this time choosing $s_\alpha= L_f(x_\alpha,\delta r)$, so that $B(x_\alpha,\delta r)\subset U(x_\alpha,f,s_\alpha)\subset B(x_\alpha,H\delta r)$. Choosing $\delta$ small enough so that 
		\begin{align*}
			\diam B(x_\alpha,H\delta r)\leq d(x_0,B(x_\alpha,H\delta r))/(8\lambda^{2n+1})
		\end{align*}
		to complete the proof.
	
	\end{proof}
	
	\begin{proof}[Proof of Theorem~\ref{thm:thm branch set}]
		If $x_0\in S_{H,R}\cap\mathcal B_f$, then by Proposition~\ref{prop:general prop}, $S_{H,R}$ cannot be $\delta$-dense at $x_0$. Hence at every sufficiently small scale $r_0$, there exists a ball of radius comparable to $r_0$, which does not intersect $S_{H,R}$. Thus it does not intersect $S_{H,R}\cap\mathcal B_f$, which gives the quantitative porosity. A similar argument gives porosity of $f(S_{H,R}\cap\mathcal B_f)$. The same proof works for $S^*_{H,R}$ as well. Taking countable unions over rational $R>0$ gives the corresponding statements for the sets where $H_f$ or $H_f^*$ being uniformly bounded.
	\end{proof}
	
		\begin{proof}[Proof of Corollary \ref{coro:Corollary 1.2}]
		By Theorem~\ref{thm:thm branch set}, each of the sets
		$S_{m,1/k}\cap\mathcal B_f$, $f(S_{m,1/k}\cap\mathcal B_f)$,
		$S^*_{m,1/k}\cap\mathcal B_f$, and $f(S^*_{m,1/k}\cap\mathcal B_f)$,
		$m,k\in\mathbb N$, is porous and hence null for every locally doubling measure. Since
		$S_f=\bigcup_{m,k\in\mathbb N}S_{m,1/k}$ and
		$S_f^*=\bigcup_{m,k\in\mathbb N}S^*_{m,1/k}$, the conclusion follows from countable subadditivity.
	\end{proof}
	
	\begin{proof}[Proof of Corollary~\ref{coro:BDD-porosity}]
		Since $f$ is a locally $L$-BDD map,  Lemma~\ref{lemma:BDD-inverse-dilatation} implies that $H_f^*(x,s)\le2CL^2$ for sufficiently small $s$. The claim follows directly from Theorem~\ref{thm:thm branch set}.
	\end{proof}

	\begin{proof}[Proof of Corollary~\ref{coro:Heinonen-Rickman}]
		Apply Lemma~\ref{lemma:BLD-linear-dilatation} and then Theorem~\ref{thm:thm branch set}. 
	\end{proof}

	\section{Geometric application II: bi-Lipschitz embeddability}\label{sec:bi-Lipschitz embeddability}

	In this section, we shall use the induced pullback metric to study bi-Lipschitz embeddability of BLD Euclidean spaces. Throughout this section, we assume $f:X\to Y$ is a branched covering of finite multiplicity
	\[
	N(f,X):=\sup_{y\in Y}\#f^{-1}(y)<\infty.
	\]
	It is convenient to pass from $d_X$ to the adapted metric $d_f$. In this case, Lemmas~\ref{lemma:df-comparison} and~\ref{lemma:adapted-1BDD}, together with Proposition~\ref{prop:BDD-implies-BLD}, show explicitly that this change is bi-Lipschitz and make the same set map
	\[
	f_{\rm adp}:(X,d_f)\longrightarrow (Y,d_Y)
	\]
	a $1$-BDD and hence a $1$-BLD branched covering.  
	
	\subsection{An auxiliary lemma}
	
	For $y\in Y$ and $r>0$, let
	\[
	\nu_f(y,r):=\#\{\text{components of }f^{-1}(B(y,r))\}.
	\]
	If $V$ is a component of $f^{-1}(B(y,r))$ so that $\overline{V}$ is contained in a normal domain for $f$, then the restriction $f|_V$ is proper and open.  Consequently, $f(V)=B(y,r)$.  
	
	For $1\le k<N(f,X)$, we  define
	\begin{equation}\label{eq:Rk-def}
		R_k(y):=\frac15\inf\left\{r>0:\nu_f(y,r)\le k\right\}.
	\end{equation}
	Set
	\[
	Y_k:=\{y\in Y:R_k(y)>0\}.
	\]
	For $y\in Y_k$, the definition and component-surjectivity imply that $f^{-1}(B(y,5R_k(y)))$ has at most $k$ components in the limiting sense, whereas $f^{-1}(B(y,r))$ has more than $k$ components for every $r<5R_k(y)$.  
	
	The following elementary stability property will be useful for our later embedding construction.
	
	\begin{lemma}\label{lemma:Rk-Lipschitz}
		For every $k<N(f,X)$ and $y,y'\in Y$,
		\[
		|R_k(y)-R_k(y')|\le \frac15 d_Y(y,y').
		\]
		In particular, $Y_k$ is open.  Moreover, if $B(y,R_k(y))\cap B(y',R_k(y'))\ne\emptyset$, then
		\[
		\frac23R_k(y)\le R_k(y')\le\frac32R_k(y),
		\]
		and if $B(y,2R_k(y))\cap B(y',2R_k(y'))\ne\emptyset$, then
		\[
		\frac37R_k(y)\le R_k(y')\le\frac73R_k(y).
		\]
	\end{lemma}
	\begin{proof}
		Note that $B(y',r-d_Y(y,y'))\subset B(y,r)$ whenever the left-hand radius is positive.   It follows from this and \eqref{eq:Rk-def} that
		\[
		5R_k(y')\le 5R_k(y)+d_Y(y,y').
		\]
		By symmetry, we also have the revise inequality. This proves the Lipschitz estimate.  The two comparison estimates follow from the above estimate with
		$d_Y(y,y')<R_k(y)+R_k(y')$ and
		$d_Y(y,y')<2R_k(y)+2R_k(y')$, respectively.
	\end{proof}
	
	\begin{lemma}\label{lemma:fiber-scale}
		Assume that $X$ and $Y$ are $1$-bounded turning and $f$ is $1$-BDD.  If $x_1\ne x_2$ and $f(x_1)=f(x_2)=y$, then for some $k\in\{1,\ldots,N(f,X)-1\}$, it holds 
		\begin{equation}\label{eq:fiber-scale}
			\frac1{10}d_X(x_1,x_2)\le R_k(y)\le\frac15d_X(x_1,x_2).
		\end{equation}
	\end{lemma}
	\begin{proof}
		Set $y=f(x_1)=f(x_2)$ and define 
		\[
		s=s(x_1,x_2;y):=\inf\bigl\{r>0:\ x_1\text{ and }x_2
		\text{ belong to the same component of }f^{-1}(B(y,r))\bigr\}.
		\]
		Since $f$ is discrete, Lemma \ref{lemma:normal-neighborhoods} implies $s>0$.  On the other hand, since $X$ has bounded
		turning, $x_1$ and $x_2$ lie in a compact continuum and thus $s<\infty$.
		
		We claim that $\frac12d_X(x_1,x_2)\le s\le d_X(x_1,x_2)$.  Indeed, if $r>s$, then the points $x_1,x_2$ lie in one component
		$V$ of $f^{-1}(B(y,r))$.  Since $V$ is connected and locally path connected, there is a compact continuum
		$E\subset V$ joining $x_1$ to $x_2$.  The $1$-BDD identity gives
		\[
		d_X(x_1,x_2)\le \diam_X E=\diam_Y f(E)\le \diam B(y,r)\le 2r.
		\]
		Letting $r\downarrow s$ gives
		\[
		\frac12 d_X(x_1,x_2)\le s.
		\]
		Conversely, since $X$ has $1$-bounded turning, for every $\varepsilon>0$, there exists a continuum
		$E_\varepsilon$ containing $x_1,x_2$ with
		\[
		\diam_X E_\varepsilon\le d_X(x_1,x_2)+\varepsilon.
		\]
		Since $f(x_1)=y$ and $f$ is $1$-BDD, we have
		\[
		f(E_\varepsilon)\subset
		B\bigl(y,d_X(x_1,x_2)+\varepsilon\bigr).
		\]
		Thus $x_1$ and $x_2$ belong to the same component of the inverse image of the ball and so
		\[
		s\le d_X(x_1,x_2)+\varepsilon.
		\]
		Letting $\varepsilon\downarrow0$ gives $s\le d_X(x_1,x_2)$. This completes the proof of claim.
		
		Note that the integer-valued function
		$r\mapsto \nu_f(y,r)$ is nonincreasing.
		At the radius $s$ the two components containing $x_1$ and $x_2$ merge and so $\nu_f(y,r)$ decreases.
		Choose an integer $k\in\{1,\ldots,N(f,X)-1\}$ between the values immediately after and immediately before $\nu_f(y,r)$ changes.
		Then, we have
		\[
		\inf\{r>0:\nu_f(y,r)\le k\}=s.
		\]
		In view of \eqref{eq:Rk-def}, this means $5R_k(y)=s$.  Combining this identity with
		the previous claim gives the desired estimate
		\[
		\frac1{10}d_X(x_1,x_2)\le R_k(y)\le\frac15d_X(x_1,x_2).
		\]
	
	\end{proof}
	
	\subsection{The embedding theorem}
	The main result of this section is the following bi-Lipschitz embedding. 
	\begin{theorem}\label{thm:biLip embedding I}
		Let $f:X\to Y$ be an $L$-BDD branched covering with $N=N(f,X)<\infty$.  Suppose that $X$ and $Y$ have bounded turning and that $Y$ is doubling.  Then there are an integer $m\le c_d(N-1)$ and a bi-Lipschitz embedding
		\[
		\Psi:X\longrightarrow Y\times\mathbb R^m,
		\]
		where $c_d$ depends only on the doubling and bounded-turning data of $Y$; the bi-Lipschitz constant of $\Psi$ depends only on $L$, the bounded-turning data of $X,Y$, and $N$.
	\end{theorem}
	
	The proof requires an elementary auxiliary lemma. 
	\begin{lemma}\label{lemma:for biLip embedding}
		Let $f:X\to Y$ be $L_f$-Lipschitz and let $\phi:X\to W$ be $A$-Lipschitz.  Assume that there are constants
	$\varepsilon>0$ and $C_0\ge1$ with the following two properties:
	\begin{enumerate}
		\item[(i)] $d_W(\phi(x),\phi(x'))\ge\varepsilon d_X(x,x')$
		 whenever $x\ne x'$ and $f(x)=f(x')$;
		 
		\item[(ii)] for every $x_2\in X$ and every $y_1\in f(X)$, there exists
		$x_1'\in f^{-1}(y_1)$ such that
		\[
		d_X(x_1',x_2)\le C_0\,d_Y(y_1,f(x_2)).
		\]
	\end{enumerate}
	Then $\Psi=(f,\phi):X\to Y\times W$ is $C$-bi-Lipschitz with $C=C(L_f,A,\varepsilon,C_0)$.
	\end{lemma}
	\begin{proof}
	Since
	\[
	d_{Y\times W}(\Psi(x_1),\Psi(x_2))
	\le \max\{L_f,A\}\,d_X(x_1,x_2),
	\]
	$\Psi$ is $\max\{L_f,A\}$-Lipschitz.
	To show the other inequality, choose
	\[
	0<\delta<\frac{\varepsilon}{C_0(A+\varepsilon)}.
	\]
	Let $x_1,x_2\in X$ and set $d=d_X(x_1,x_2)$.  If
	$d_Y(f(x_1),f(x_2))\ge\delta d$, then the conclusion clearly holds.
	
	Next, we assume	$d_Y(f(x_1),f(x_2))<\delta d$.  Applying (ii) with $y_1=f(x_1)$ and $x_2$, we obtain
	$x_1'\in f^{-1}(f(x_1))$ with $d_X(x_1',x_2)\le C_0\delta d$.	Hence
	\[
	d_X(x_1,x_1')\ge (1-C_0\delta)d.
	\]
	By (i) and the $A$-Lipschitz property of $\phi$, we have
	\[
	\begin{split}
		\varepsilon(1-C_0\delta)d
		&\le d_W(\phi(x_1),\phi(x_1'))\le d_W(\phi(x_1),\phi(x_2))
		+A\,d_X(x_2,x_1')\\
		&\le d_W(\phi(x_1),\phi(x_2))+AC_0\delta d.
	\end{split}
	\]
	Thus
	\[
	\bigl(\varepsilon(1-C_0\delta)-AC_0\delta\bigr)d
	\le d_W(\phi(x_1),\phi(x_2)).
	\]
	The coefficient on the left is positive by the choice of $\delta$.  This proves the desired
	bi-Lipschitz lower bound.
	\end{proof}
	
	\begin{proof}[Proof of Theorem~\ref{thm:biLip embedding I}]
		We divide the proof into a few steps.
		\smallskip 
		
		\textbf{Step 1: Normalization.}
		\smallskip
		
		By Lemma~\ref{lemma:df-comparison}, replacing $d_X$ by $d_f$ changes the metric bi-Lipschitzly.
		As $Y$ has bounded turning, replacing $d_Y$ by its Mazurkiewicz metric also changes the target metric bi-Lipschitzly.  Thus it suffices to prove the theorem when
		\[
		X\text{ and }Y\text{ are }1\text{-bounded turning}
		\quad\text{and}\quad
		f\text{ is }1\text{-BDD}.
		\]
		Then Proposition~\ref{prop:BDD-implies-BLD} gives that $f$ is $1$-BLD.
		
		We shall also use the following elementary consequence of component-surjectivity.  If
		$y_1\in B(f(x),r)$ and $V$ is the component of $f^{-1}(B(f(x),r))$ containing $x$, then there is
		$x'\in V$ with $f(x')=y_1$.  A path in $V$ from $x$ to $x'$ has image in $B(f(x),r)$ and thus
		\[
		d_X(x,x')=d_f(x,x')\le 2r.
		\]
		Letting $r\downarrow d_Y(f(x),y_1)$ shows that 
		Lemma~\ref{lemma:for biLip embedding} (ii) holds with $C_0=2$.
		\smallskip 
		
		\textbf{Step 2: Construct a good covering.}
		\smallskip 
		
		Fix $k\in\{1,\ldots,N-1\}$ and write $R(y)=R_k(y)$ and $Y_k=\{y\in Y:R(y)>0\}$.
		Choose a maximal set $T^k\subset Y_k$ so that for $y\ne y'$, $y,y'\in T^k$, it holds
		\[
		d_Y(y,y')\ge\frac12\min\{R(y),R(y')\}.
		\]
		Maximality implies that the balls $B_y:=B(y,R(y))$, $y\in T^k$, cover $Y_k$.  Indeed, if $z\in Y_k$ were not covered, then $z$ could be added to $T^k$ without breaking the
		separation condition.
		
		Consider the intersection graph of the doubled balls $2B_y$.  If $2B_y\cap2B_{y'}\ne\emptyset$, then
		Lemma~\ref{lemma:Rk-Lipschitz} gives
		\[
		\frac37R(y)\le R(y')\le\frac73R(y).
		\]
		The centers of all doubled balls meet a fixed $2B_y$ and thus lie in a ball of radius comparable to $R(y)$.
       By the defining separation of $T^k$, they are mutually separated at a scale comparable to $R(y)$.
		Since $Y$ is doubling, the intersection graph has degree at most
		$c_d-1$, where $c_d$ depends only on the doubling constant of $Y$. Therefore, we have
		\[
		T^k=T_1^k\cup\cdots\cup T_{c_d}^k
		\]
		such that, for fixed $j$, the doubled balls $2B_y$, $y\in T_j^k$, are pairwise disjoint.
		
		For $y\in T_j^k$, let
		\[
		U_{y,1}^k,\ldots,U_{y,m(y)}^k
		\]
		be the components of $f^{-1}(2B_y)$ which meet $f^{-1}(B_y)$.
		Component-surjectivity implies that each such component contains a point of the fiber $f^{-1}(y)$ and thus $m(y)\le N$.  Assign to these components distinct labels
		$a\in\{1,\ldots,N\}$.
		
		For $x\in X$, define
		\[
		\theta_{j,k}(x)
		=
		\sum_{y\in T_j^k}\sum_{a=1}^{m(y)}
		a\,\min\{\delta_{U_{y,a}^k}(x),R(y)\}.
		\]
		For a fixed pair $(j,k)$, the supports appearing in this sum are pairwise disjoint.  Hence at each point at most
		one summand is nonzero.  Since each truncated function is $1$-Lipschitz and each label is at most $N$,
		$\theta_{j,k}$ is $N$-Lipschitz.  Set
		\[
		\phi=
		(\theta_{1,1},\ldots,\theta_{c_d,1},
		\theta_{1,2},\ldots,\theta_{c_d,N-1})
		:X\longrightarrow\mathbb R^{c_d(N-1)}.
		\]
		Thus $\phi$ is quantitatively Lipschitz.
		\smallskip
		
		\textbf{Step 3: Separation of two points in the same fiber.}
		\smallskip
		
		Let $x_1\ne x_2$ and $f(x_1)=f(x_2)=y$.  By Lemma~\ref{lemma:fiber-scale}, for some $k$, we have
		\[
		5R_k(y)=s(x_1,x_2;y),\qquad
		\frac1{10}d_X(x_1,x_2)\le R_k(y)\le\frac15d_X(x_1,x_2).
		\]
		Choose $y'\in T_j^k$ with $y\in B_{y'}$.  Since $B_y^k\cap B_{y'}^k\neq \emptyset$,
		Lemma~\ref{lemma:Rk-Lipschitz} gives
		\[
		\frac23R_k(y)\le R_k(y')\le\frac32R_k(y).
		\]
		Moreover, we have 
		\[
		2B(y',R_k(y'))\subset
		B\bigl(y,3R_k(y')\bigr)
		\subset B\bigl(y,\tfrac92R_k(y)\bigr).
		\]
		Since the last radius is strictly smaller than the merging radius $5R_k(y)$,  $x_1$ and $x_2$ lie in
		distinct components, say $U_{y',a_1}^k$ and $U_{y',a_2}^k$, of $f^{-1}(2B_{y'})$. This implies
		$a_1\ne a_2$.
		
		We next show that both points lie at depth at least $R_k(y')$ in their components.  Suppose
		a continuum joins $x_1$ to $X\setminus U_{y',a_1}^k$.  Such a continuum cannot remain in
		$f^{-1}(2B_{y'})$, since otherwise it would connect two different components of the inverse image.  Hence its image
		meets $Y\setminus2B_{y'}$.  Since $f(x_1)=y\in B_{y'}$, the image has diameter at least $R_k(y')$.
		The $1$-BDD identity then gives
		\[
		\delta_{U_{y',a_1}^k}(x_1)\ge R_k(y').
		\]
		The same argument applies to $x_2$.  Thus $\theta_{j,k}(x_i)=a_iR_k(y')$, $i=1,2$.		Therefore
		\[
		|\theta_{j,k}(x_1)-\theta_{j,k}(x_2)|
		\ge R_k(y')
		\ge \frac23R_k(y)
		\ge \frac1{15}d_X(x_1,x_2).
		\]
		This proves Lemma~\ref{lemma:for biLip embedding} (i) with	$\varepsilon=1/15$.
		\smallskip 
		
		\textbf{Step 5: Final conclusion.}
		\smallskip 
		
		Apply Lemma~\ref{lemma:for biLip embedding} to $\Psi=(f,\phi)$ gives a quantitative bi-Lipschitz embedding
		\[
		\Psi:X\longrightarrow Y\times\mathbb R^{c_d(N-1)}.
		\]
		This completes the proof of Theorem~\ref{thm:biLip embedding I}.
	\end{proof}
	
	\begin{corollary}\label{coro:biLip embedding II}
		If, in addition, $Y$ admits a bi-Lipschitz embedding into some Euclidean space, then so does $X$.  In particular, a finite-multiplicity BLD covering of a locally compact quasiconvex metric space over $\mathbb R^n$ has a locally bi-Lipschitz Euclidean embedding, with local target dimension controlled by $n$ and the local multiplicity.
	\end{corollary}
	\begin{proof}
		Let $J:Y\to\mathbb R^q$ be a $C$-bi-Lipschitz embedding.  If
	$\Psi=(f,\phi):X\to Y\times\mathbb R^m$ is the embedding provided by
	Theorem~\ref{thm:biLip embedding I}, then
	\[
	x\longmapsto (J(f(x)),\phi(x))\in\mathbb R^{q+m}
	\]
	is bi-Lipschitz, with constants depending only on those of $J$ and $\Psi$.  This proves the first assertion.
	
	For the local statement, let $f:X\to\mathbb R^n$ be a finite-multiplicity BLD covering and let $x\in X$.
	Choose a relatively compact normal neighborhood $U\ni x$ so that the multiplicity is finite on $U$.  Since locally
	quasiconvex spaces are locally bounded turning, Proposition~\ref{prop:BLD-to-BDD-finite} implies that, after shrinking $U$ if necessary, $f|_U$ is quantitatively BDD.  For Euclidean target, Theorem~\ref{thm:biLip embedding I} applied to $f|_U$  gives a
	bi-Lipschitz embedding of a smaller neighborhood of $x$ into
	\[
	\mathbb R^n\times\mathbb R^m\cong\mathbb R^{n+m},
	\]
	where $m$ is bounded by the local multiplicity and the quantitative local data.  Since $x$ was arbitrary,
	$X$ is locally bi-Lipschitz Euclidean in a finite ambient dimension.
	\end{proof}

	\subsection{The induced metric and BLD-Euclidean examples}
	
	Following Heinonen--Rickman \cite{hr02}, we call an oriented generalized $n$-manifold $Z$ \emph{locally BLD-Euclidean} if every point has a neighborhood admitting a BLD branched covering onto a Euclidean domain.  
	
	Our next result shows that this property can be produced from a purely topological finite-to-one branched cover using the induced pullback metric.
	
	\begin{proposition}\label{prop:induced-BLD-Euclidean}
		Let $g:Z\to M$ be a finite-to-one branched covering onto a Riemannian $n$-manifold.  Endow $Z$ with the metric $d_g$ induced by the image-diameter structure of $M$.  Then $d_g$ induces the original topology of $Z$ and $(Z,d_g)$ is locally BLD-Euclidean. Moreover, the same set map
		\[
		g_{\rm adp}:(Z,d_g)\to M
		\]
		is $1$-BDD and $1$-BLD.
	\end{proposition}
	\begin{proof}
	The metric $d_g$ is defined as
	\[
	d_g(z,z')=\inf\{\diam_M g(E):E\subset Z\text{ is a compact continuum containing }z,z'\}.
	\]
	Theorem~\ref{prop:topological-cover-realization} gives that this is nondegenerate and  its metric topology coincides with the original manifold topology on $Z$.  
	Moreover, for every continuum $E\subset Z$, we have 
	\[
	\diam_{d_g}E=\diam_M g(E).
	\]
	This is precisely the $1$-BDD condition.  Proposition~\ref{prop:BDD-implies-BLD} then gives
$g:(Z,d_g)\to M$ is $1$-BLD.

	Fix $z\in Z$.  Choose a normal neighborhood
	$U=U(z,g,r)$ with $g(U)=B(g(z),r)$.  Choose a sufficiently small bi-Lipschitz Riemannian coordinate chart
	\[
	\chi:B(g(z),r)\longrightarrow\Omega\subset\mathbb R^n.
	\]
	Since $g|_U$ is $1$-BLD and $\chi$ is bi-Lipschitz, the map
	\[
	\chi\circ g:(U,d_g)\longrightarrow\Omega
	\]
	is BLD.  Thus every point of $(Z,d_g)$ has a neighborhood admitting a BLD branched covering onto a Euclidean
	domain, which is by definition locally BLD-Euclidean.
	\end{proof}
	
	\subsection{Natural examples and counterexamples}\label{subsec:some natural examples}
	
	 Let $W:S^n\to S^n$ be the standard degree-two winding branched covering, whose branch image is a standard embedded $(n-2)$-sphere $\Sigma$.  Choose a closed $(n-2)$-disk $D\subset\Sigma$, let $\sigma$ denote normalized spherical $n$-measure on $S^n$ and $\tau$ the normalized $\mathcal H^{n-2}$-measure on $D$.  The measure
		\[
		\mu=\tfrac12\sigma+\tfrac12\tau
		\]
		is non-atomic and positive on every nonempty open set.  By the Oxtoby--Ulam Theorem~\cite{ou41}, there exists a homeomorphism $H:S^n\to S^n$ satisfying $H_*\mu=\sigma$.  Hence
		\[
		\sigma(H(D))=\mu(D)\ge\tfrac12,
		\]
		and therefore the degree-two branched covering
		\[
		F=H\circ W:S^n\to S^n
		\]
		has branch image of positive $n$-dimensional measure.
	
	\begin{lemma}\label{lemma:positive-measure-topological-cover}
		For every $n\ge3$ there exists a degree-two proper branched covering
			\[
			g:\mathbb R^n\to\mathbb R^n
			\]
			for which $\mathcal H^n(g(\mathcal B_g))>0$.
	\end{lemma}
	\begin{proof}
		Let $F=H\circ W$ be the sphere map above.  Choose a branch value $b\in F(\mathcal B_F)$.  For the standard degree-two winding map, a branch value has a single preimage and so we may write $F^{-1}(b)=\{a\}$. Restrict $F$ to the punctured spheres:
			\[
			g=F|_{S^n\setminus\{a\}}:S^n\setminus\{a\}\longrightarrow S^n\setminus\{b\}.
			\]
			
			Identify the two punctured spheres with $\mathbb R^n$ by stereographic projections.  Since
			$F(\mathcal B_F)\setminus\{b\}$ has positive spherical $n$-measure, we may choose a compact
			subset $K$ with positive spherical $n$-measure. Since stereographic projection is bi-Lipschitz on a neighborhood of
			$K$, its Euclidean image has positive $\mathcal H^n$-measure.  Therefore, the proper degree-two
			map $g:\mathbb R^n\to\mathbb R^n$ satisfies
			\[
			\mathcal H^n(g(\mathcal B_g))>0.
			\]
	\end{proof}
	
	\begin{theorem}\label{thm:counter-example branch set I}
		Let $g:Z\to M$ be a finite-to-one topological branched covering onto a Riemannian $n$-manifold and suppose that $\mathcal H^n(g(\mathcal B_g))>0$.  Equip $Z$ with the canonical diameter metric $d_g$.  Then $g:(Z,d_g)\to M$ is $1$-BDD and $1$-BLD, $(Z,d_g)$ is locally BLD-Euclidean with 
		$$\mathcal H^n_{d_g}(\mathcal B_g)>0.$$		
		If $M=\mathbb R^n$ and $g$ has finite multiplicity, then $(Z,d_g)$ admits a bi-Lipschitz embedding into a finite-dimensional Euclidean space.  Moreover,  $(Z,d_g)$ is locally Ahlfors $n$-regular and supports a local $(1,1)$-Poincar\'e inequality.
	\end{theorem}
	\begin{proof}
			Equip $Z$ with the canonical metric $d_g$.
		Proposition~\ref{prop:induced-BLD-Euclidean} shows that $d_g$ induces the original topology on $Z$, $g:(Z,d_g)\to M$ is $1$-BDD and $1$-BLD and $(Z,d_g)$ is locally BLD-Euclidean.

		The exact diameter identity implies that $g$ is $1$-Lipschitz:
		\[
		d_M(g(z),g(z'))\le d_g(z,z').
		\]
		Consequently, for every $E\subset Z$, we have 
		\[
		\mathcal H^n_M(g(E))\le \mathcal H^n_{d_g}(E).
		\]
	    Since $\mathcal H^n_M(g(\mathcal B_g))>0$, 	taking $E=\mathcal B_g$ above gives $\mathcal H^n_{d_g}(\mathcal B_g)>0$.

		By the definition of locally BLD-Euclidean, every point has a neighborhood carrying a finite-multiplicity BLD
		map onto a Euclidean domain.  The standard local structure theorems for BLD-Euclidean generalized manifolds \cite{hr02} imply local Ahlfors $n$-regularity and a local
		$(1,1)$-Poincar\'e inequality.  In particular, these conclusions are local and their constants depend only on the
		local BLD and multiplicity data.
		
		Assume now that $M=\mathbb R^n$ and that $g$ has finite multiplicity.  Since
		$(d_g)_M=d_g$ by Theorem~\ref{thm:canonical-diameter-metric}, $(Z,d_g)$ has bounded turning.
		The target $\mathbb R^n$ is doubling and bounded turning.  Hence Theorem~\ref{thm:biLip embedding I} applies and
		gives a bi-Lipschitz embedding
		\[
		(Z,d_g)\hookrightarrow\mathbb R^n\times\mathbb R^m
		\cong\mathbb R^{n+m}
		\]
		for some finite $m$.

	\end{proof}
	
	\begin{proof}[Proof of Theorem~\ref{thm:counter-example}]
		Fix $n\ge3$.  Lemma~\ref{lemma:positive-measure-topological-cover} gives a proper degree-two branched covering $g:\mathbb R^n\to\mathbb R^n$	with
	\[
	\mathcal H^n(g(\mathcal B_g))>0.
	\]
	Equip the domain with the canonical induced metric $d_g$ and write
	\[
	X:=(\mathbb R^n,d_g),\qquad f:=g.
	\]
	By Proposition~\ref{prop:topological-cover-realization}, $d_g$ induces the ordinary Euclidean topology, and so $X$ is homeomorphic to $\mathbb R^n$.  Since both $g$ and the Euclidean target are proper, $X$ is proper.
	
	By Theorem~\ref{thm:counter-example branch set I},	$f$ is $1$-BDD and $1$-BLD; $X$ is locally Ahlfors $n$-regular, supports a local $(1,1)$-Poincar\'e inequality and admits a bi-Lipschitz embedding into a finite-dimensional Euclidean space.
	Moreover, since $f:(X,d_g)\to\mathbb R^n$ is $1$-Lipschitz and $\mathcal H^n_{\mathbb R^n}(f(\mathcal B_f))>0$, we have
	\[
	\mathcal H^n_X(\mathcal B_f)
	\ge
	\mathcal H^n_{\mathbb R^n}(f(\mathcal B_f))>0,
	\]
	This completes the proof of Theorem~\ref{thm:counter-example}.
	\end{proof}
	
	\begin{proof}[Proof of Corollary~\ref{coro:Heinonen-Semmes}]
	Let $F=H\circ W:S^n\to S^n$ be the degree-two branched covering constructed in the proof of
	Lemma~\ref{lemma:positive-measure-topological-cover}.  By construction,
	\[
	\mathcal H^n_{S^n}(F(\mathcal B_F))>0.
	\]
	Suppose, towards a contradiction, that there is a compatible metric $d$ on the domain sphere such that
	$(S^n,d)$ is Ahlfors $n$-regular and linearly locally contractible and
	\[
	F:(S^n,d)\to(S^n,d_{\rm round})\qquad \text{is BLD.}
	\]

	Note that Ahlfors $n$-regularity implies local doubling.
	Linear local contractibility means that there is $\lambda\ge1$ such that every sufficiently small ball
	$B_d(x,r)$ contracts inside $B_d(x,\lambda r)$.  In particular, for every $0\le k\le n$, every map
	$S^k\to B_d(x,r)$ is null-homotopic in $B_d(x,\lambda r)$ and hence the domain is locally linearly locally	$n$-connected.  The round sphere is doubling and bounded
	turning.  Finally, by Lemma~\ref{lemma:BLD-linear-dilatation}, a BLD map has uniformly bounded  linear dilatation on sufficiently small scales. Thus all the assumptions of Theorem \ref{thm:thm branch set}  are verified. It follows that
	$F(\mathcal B_F)$ is quantitatively porous in the round sphere.  Since $S^n$ is Ahlfors
	$n$-regular, we have
	\[
	\mathcal H^n_{S^n}(F(\mathcal B_F))=0,
	\]
	contradicting the construction of $F$.  Hence no such compatible metric $d$ exists.
	\end{proof}

\end{document}